\documentclass[oneside, 11pt]{amsart} 
\usepackage{amsmath,amsthm,amssymb,epic}
\usepackage{eqlist,eqparbox}
\usepackage{amsfonts}
\usepackage{latexsym}
\usepackage[2emode]{psfrag}
\usepackage{amsthm}
\usepackage{amsmath}
\usepackage[all]{xy}

\begin{document}  

\newcommand{\norm}[1]{\| #1 \|}
\def\N{\mathbb N}
\def\Z{\mathbb Z}
\def\Q{\mathbb Q}
\def\mod{\textit{\emph{~mod~}}}
\def\R{\mathcal R}
\def\S{\mathcal S}
\def\*  C{{*  \mathcal C}} 
\def\C{\mathcal C}
\def\D{\mathcal D}
\def\J{\mathcal J}
\def\M{\mathcal M}
\def\T{\mathcal T}          

\newcommand{\Hom}{{\rm Hom}}
\newcommand{\End}{{\rm End}}
\newcommand{\Ext}{{\rm Ext}}
\newcommand{\Mor}{{\rm Mor}\,}
\newcommand{\Aut}{{\rm Aut}\,}
\newcommand{\Hopf}{{\rm Hopf}\,}
\newcommand{\Ann}{{\rm Ann}\,}
\newcommand{\Ker}{{\rm Ker}\,}
\newcommand{\Reg}{{\rm Reg}\,}
\newcommand{\Den}{{\rm Den}\,}
\newcommand{\Coker}{{\rm Coker}\,}
\newcommand{\Img}{{\rm Im}\,}
\newcommand{\coim}{{\rm Coim}\,}
\newcommand{\Trace}{{\rm Trace}\,}
\newcommand{\Char}{{\rm Char}\,}
\newcommand{\Mod}{{\rm mod}}
\newcommand{\Spec}{{\rm Spec}\,}
\newcommand{\sgn}{{\rm sgn}\,}
\newcommand{\Id}{{\rm Id}\,}
\newcommand{\Com}{{\rm Com}\,}
\newcommand{\codim}{{\rm codim}}
\newcommand{\Mat}{{\rm Mat}}
\newcommand{\can}{{\rm can}}
\newcommand{\sign}{{\rm sign}}
\newcommand{\kar}{{\rm kar}}
\newcommand{\rad}{{\rm rad}}
\newcommand{\ra}{\rightarrow}
\newcommand{\rs}{\rightsquigarrow}
\def\lan{\langle}
\def\ran{\rangle}
\def\ot{\otimes}

\def\id{{\small \textit{\emph{1}}}\!\!1}    
\def\To{{\multimap\!\to}}
\def\bigperp{{\LARGE\textrm{$\perp$}}} 
\newcommand{\QED}{\hspace{\stretch{1}}
\makebox[0mm][r]{$\Box$}\\}

\def\RR{{\mathbb R}}
\def\FF{{\mathbb F}}
\def\NN{{\mathbb N}}
\def\CC{{\mathbb C}}
\def\DD{{\mathbb D}}
\def\ZZ{{\mathbb Z}}
\def\QQ{{\mathbb Q}}
\def\HH{{\mathbb H}}
\def\units{{\mathbb G}_m}
\def\GG{{\mathbb G}}
\def\EE{{\mathbb E}}
\def\FF{{\mathbb F}}
\def\rightact{\hbox{$\leftharpoonup$}}
\def\leftact{\hbox{$\rightharpoonup$}}

\newcommand{\Aa}{\mathcal{A}}
\newcommand{\Bb}{\mathcal{B}}
\newcommand{\Cc}{\mathcal{C}}
\newcommand{\Dd}{\mathcal{D}}
\newcommand{\Ee}{\mathcal{E}}
\newcommand{\Ff}{\mathcal{F}}
\newcommand{\Hh}{\mathcal{H}}
\newcommand{\Ii}{\mathcal{I}}
\newcommand{\Mm}{\mathcal{M}}
\newcommand{\Pp}{\mathcal{P}}
\newcommand{\Rr}{\mathcal{R}}
\def\*  C{{}*  \hspace*  {-1pt}{\Cc}}

\def\text#1{{\rm {\rm #1}}}

\def\smashco{\mathrel>\joinrel\mathrel\triangleleft}
\def\cosmash{\mathrel\triangleright\joinrel\mathrel<}

\def\Nat{\dul{\rm Nat}}

\renewcommand{\subjclassname}{\textup{2000} Mathematics Subject
     Classification}

\newtheorem{prop}{Proposition}[section] 
\newtheorem{lemma}[prop]{Lemma}
\newtheorem{cor}[prop]{Corollary}
\newtheorem{theo}[prop]{Theorem}
                       
\theoremstyle{definition}
\newtheorem{Def}[prop]{Definition}
\newtheorem{ex}[prop]{Example}
\newtheorem{exs}[prop]{Examples}
\newtheorem{Not}[prop]{Notation}
\newtheorem{Ax}[prop]{Axiom}
\newtheorem{rems}[prop]{Remarks}
\newtheorem{rem}[prop]{Remark}
\newtheorem{op}[prop]{Open problem}
\newtheorem{conj}[prop]{Conjecture}

\def\smashco{\mathrel>\joinrel\mathrel\triangleleft}
\def\curlarrow{\mathrel\sim\joinrel\mathrel>}

\title{Fantastic deductive systems in probability theory on generalizations of fuzzy structures}

\author[Lavinia Corina Ciungu]{Lavinia Corina Ciungu} 

\begin{abstract} 
The aim of this paper is to introduce the notion of fantastic deductive systems on generalizations of fuzzy 
structures, and to emphasize their role in the probability theory on these algebras. 
We give a characterization of commutative pseudo-BE algebras and we generalize an axiom system consisting of four identities to the case of commutative pseudo-BE algebras. 
We define the fantastic deductive systems of pseudo-BE algebras and we investigate their 
properties. It is proved that, if a pseudo-BE(A) algebra $A$ is commutative, then all deductive systems of $A$ are fantastic. 
Moreover, we generalize the notions of measures, state-measures and measure-morphisms to the case of 
pseudo-BE algebras and we also prove that there is a one-to-one correspondence between the set of all Bosbach 
states on a bounded pseudo-BE algebra and the set of its state-measures. 
The notions of internal states and state-morphism operators on pseudo-BCK algebras are extended to the case of pseudo-BE algebras and we also prove that any type II state operator on a pseudo-BE algebra is a state-morphism 
operator on it. 
The notions of pseudo-valuation and commutative pseudo-valuation on pseudo-BE algebras are defined and investigated.  For the case of commutative pseudo-BE algebras we prove that the two kind of pseudo-valuations coincide.   
Characterizations of pseudo-valuations and commutative pseudo-valuations are given.
We show that the kernel of a Bosbach state (state-morphism, measure, type II state operator, pseudo-valuation) is a fantastic deductive system. \\

\textbf{AMS classification (2000):} 03G25, 06F35, 003B52
\end{abstract}

\maketitle

\section{Introduction}

Developing probabilistic theories on algebras of multiple-valued logics is a central topic in the study of 
fuzzy systems.  
Starting from the systems of positive implicational calculus, weak systems of positive implicational calculus and BCI and BCK systems, in 1966 Y. Imai and K. Is$\rm\acute{e}$ki introduced the BCK-algebras (\cite{Imai}). 
BCK-algebras are also used in a dual form, with an implication $\to$ and with one constant element $1$, that is the greatest element. Dual BCK-algebras were defined in \cite{Kim2}. 
BE-algebras have been defined in \cite{Kim1} as a generalization of BCK-algebras, and they have intensively been studied by many authors (\cite{Meng1}, \cite{Ahn2}, \cite{Rez4}, \cite{Wal2}). 
Commutative BE-algebras have been defined in \cite{Wal1} and investigated in \cite{Ahn1}, \cite{Rez2}, \cite{Cilo1}.  
It was proved in \cite{Wal1} that any dual BCK-algebra is a BE-algebra and any commutative BE-algebra is a dual 
BCK-algebra.  
Pseudo-BCK algebras were introduced by G. Georgescu and A. Iorgulescu in \cite{Geo15} as algebras 
with "two differences", a left- and right-difference, and with a constant element $0$ as the least element. Nowadays pseudo-BCK algebras are used in a dual form, with two implications, $\ra$ and $\rs$ and with one constant element $1$, that is the greatest element. Thus such pseudo-BCK algebras are in the "negative cone" and are also called "left-ones". Pseudo-BCK algebras were intensively studied in \cite{Ior14}, \cite{Ciu2}, \cite{Ior1}, \cite{Kuhr6}.  
Commutative pseudo-BCK algebras were originally defined by G. Georgescu and A. Iorgulescu in \cite{Geo15} 
under the name of \emph{semilattice-ordered pseudo-BCK algebras} and properties of these structures were investigated 
in \cite{Kuhr2} and \cite{Ciu7}. 
Pseudo-BE algebras were introduced in \cite{Bor2} as generalizations of BE-algebras and properties of these structures have recently been studied in \cite{Rez1} and \cite{Bor4}.   
Commutative pseudo-BE algebras were defined and investigated in \cite{Ciu33}. 
It was proved that the class of commutative pseudo-BE algebras is equivalent to the class of commutative 
pseudo-BCK algebras.  
Based on this result, all results holding for commutative pseudo-BCK algebras also hold for commutative 
pseudo-BE algebras. 
Introduced by W. Dudek and Y.B. Jun (\cite{Dud1}), pseudo-BCI algebras are pseudo-BCK algebras having no greatest element. Recently the pseudo-CI algebras were defined and studied in \cite{Rez7} as pseudo-BE algebras without having a greatest element. \\
Commutative ideals in BCK-algebras were introduced in \cite{Meng2}, and they were generalized in \cite{Meng3} and 
\cite{Rez2} for the case of BCI-algebras and BE-algebras, respectively. 
Commutative filters were investigated under the name of fantastic filters in \cite{Rac3}, \cite{Rac4} and \cite{Alav1} for bounded R$\ell$-monoids and pseudo-hoop algebras.  
Commutative deductive systems of pseudo-BCK algebras have been defined and studied in \cite{Ciu7}, while 
commutative pseudo-BCI algebras and commutative pseudo-BCI filters were investigated in \cite{Lu1}.  
The class of commutative ideals and deductive systems proved to play an important role in the study of internal 
states on BCK-algebras and pseudo-BCK algebras (see \cite{Bor1}, \cite{Ciu7}). 
Involutive filters were defined and studied in \cite{Zhu1} for commutative residuated lattices under the name of regular filters and they were investigated in \cite{Rac5} for the case of non-commutative residuated lattices. \\
In analogy to probability measure, the states on multiple-valued logic algebras proved to be the most suitable models for averaging the truth-value in their corresponding logics. 
States or measures give a probabilistic interpretation of randomness of events of given algebraic structures. 
For MV-algebras, Mundici introduced states (an analogue of probability measures) in 1995,
\cite{Mun1}, as averaging of the truth-value in \L ukasiewicz logic. 
The notion of a Bosbach state has been studied for other algebras of fuzzy structures such as pseudo-BL algebras, 
$R\ell$-monoids, residuated lattices, pseudo-hoops and pseudo-BCK algebras.
States on BE-algebras have been defined in \cite{Bor3}, while the states on pseudo-BE algebras were investigated in  \cite{Rez6}. Measures and internal states on pseudo-BCK algebras were studied in \cite{Ciu25} and \cite{Ciu7}, respectively. Pseudo-valuations on Hilbert algebras were defined and investigated in \cite{Bus1}, \cite{Bus2}, while  
the notions of pseudo-valuations on residuated lattices have been introduced in \cite{Bus3}. 
The concept of valuations on BE-algebras was defined and investigated in \cite{Lee1}. \\ 

The aim of this paper is to introduce the notion of fantastic deductive systems on pseudo-BE algebras, and to 
emphasize their role in the probability theory on these algebras. \\
Moreover, we generalize to the case of pseudo-BE algebras the notions of measures, state-measures and measure-morphisms investigated in \cite{Ciu25} for the case of pseudo-BCK algebras. We show that any measure morphism on a pseudo-BE algebra is a measure, and the kernel of a measure is a normal deductive system. 
If $A$ is a bounded pseudo-BE(A) algebra, then we prove that there is a one-to-one correspondence between the set 
of all Bosbach states $s$ on $A$ with $s(0)=0$ and the set of all state-measures on $A$.
The notions of internal states and state-morphism operators on pseudo-BCK algebras are extended to the case of pseudo-BE algebras and some of their properties are investigated. 
We also prove that any type II state operator on a pseudo-BE algebra is a state-morphism operator on it. \\ 
We give a characterization of commutative pseudo-BE algebras and we generalize to the case of 
commutative pseudo-BE algebras the axiom system given in \cite{Rez2} for commutative BE-algebras. 
It is proved that in the case of a commutative pseudo-BE algebra the two types of internal states coincide.
Moreover, if a commutative pseudo-BE algebra is linearly ordered, then any internal state is also a 
state-morphism operator. \\ 
We define the fantastic deductive systems of pseudo-BE algebras and we investigate their 
properties. It is proved that, if a pseudo-BE(A) algebra $A$ is commutative, then all deductive systems of $A$ are fantastic. We show that the kernel of a Bosbach state (state-morphism, measure, type II state operator) is a fantastic 
deductive system. \\ 
We introduce the notion of an involutive deductive system of a bounded pseudo-BE algebra and we prove that every 
fantastic deductive system on a bounded pseudo-BE algebra is an involutive deductive system. 
It is also proved that the kernel of a Bosbach state on a bounded pseudo-BE algebra is an involutive deductive 
system. 
The notions of pseudo-valuation and commutative pseudo-valuation on pseudo-BE algebras are defined and investigated. Given a pseudo-BE algebra $A$, it is proved that the kernel of a commutative pseudo-valuation on $A$ is a fantastic deductive system of $A$. If moreover $A$ is commutative, then we prove that any pseudo-valuation on $A$ is commutative. 
Characterizations of pseudo-valuations and commutative pseudo-valuations are given.
We study the relationships between the pseudo-valuations of homomorphic and isomorphic pseudo-BE algebras and 
the relationships between their kernels.

$\vspace*{5mm}$

\section{Preliminaries}

In this section we recall some basic notions and results regarding pseudo-BE algebras: properties, examples, 
deductive systems, congruences, homomorphisms.   

\begin{Def} \label{psBE-10} $\rm($\cite{Kuhr6}$\rm)$
A structure $(A,\rightarrow,\rightsquigarrow,1)$ of the type $(2,2,0)$ is a \emph{pseudo-BCK algebra} 
(more precisely, \emph{reversed left-pseudo-BCK algebra}) iff it satisfies the following identities and 
quasi-identity, for all $x, y, z \in A$:\\ 
$(psBCK_1)$ $(x \rightarrow y) \rightsquigarrow [(y \rightarrow z) \rightsquigarrow (x \rightarrow z)]=1;$ \\
$(psBCK_2)$ $(x \rightsquigarrow y) \rightarrow [(y \rightsquigarrow z) \rightarrow (x \rightsquigarrow z)]=1;$ \\
$(psBCK_3)$ $1 \rightarrow x = x;$ \\
$(psBCK_4)$ $1 \rightsquigarrow x = x;$ \\
$(psBCK_5)$ $x \rightarrow 1 = 1;$ \\
$(psBCK_6)$ $(x\rightarrow y =1$ and $y \rightarrow x=1)$ implies $x=y$. 
\end{Def}

The partial order $\le$ is defined by $x \le y$ iff $x \rightarrow y =1$ (iff $x \rightsquigarrow y =1$). \\
For details regarding pseudo-BCK algebras we refer the reader to \cite{Ior1}, \cite{Ior14}, \cite{Ciu2}, 
\cite{Kuhr2}. \\
Let $(A,\ra, \rs,1)$ be a pseudo-BCK algebra. Denote: \\ 
$\hspace*{2cm}$ $x\vee_1 y=(x\ra y)\rs y$ \\
$\hspace*{2cm}$ $x\vee_2 y=(x\rs y)\ra y$, \\
for all $x, y\in A$. \\ 
If $\ra = \rs$ then the pseudo-BCK algebra $(A,\ra, \rs,1)$ is a BCK-algebra and \\ 
$\hspace*{2cm}$ $x\vee y=(x\ra y)\ra y$, \\
for all $x, y\in A$. \\

\begin{prop} \label{psBE-20} $\rm($\cite{Ior1}$\rm)$ Let $(A,\rightarrow,\rightsquigarrow,1)$ be a pseudo-BCK algebra. 
Then the following hold for all $x, y, z\in A$: \\
$(1)$ $x\rightarrow (y\rightsquigarrow z)=y\rightsquigarrow (x\rightarrow z);$ \\
$(2)$ $x\rightsquigarrow (y\rightarrow z)=y\rightarrow (x\rightsquigarrow z);$ \\
$(3)$ $x \leq y$ implies $y \rightarrow z \leq x \rightarrow z$ and
      $y \rightsquigarrow z \leq x \rightsquigarrow z;$ \\
$(4)$ $x \leq y$ implies $z \rightarrow x \leq z \rightarrow y$ and
      $z\rightsquigarrow x \leq z \rightsquigarrow y;$ \\
$(5)$ $x\rightarrow y\le (z\rightarrow x)\rightarrow (z\rightarrow y)$ and 
      $x\rightsquigarrow y\le (z\rightsquigarrow x)\rightsquigarrow (z\rightsquigarrow y);$ \\
$(6)$ $x\vee_1 y\ra y=x\ra y$ and $x\vee_2 y\rs y=x\rs y$.                         
\end{prop}        
                         
Pseudo-BE algebras were introduced in \cite{Bor2} as generalizations of BE-algebras and properties of these 
structures have recently been studied in \cite{Rez1} and \cite{Bor4}. 

\begin{Def} \label{psBE-30} $\rm($\cite{Bor2}$\rm)$
A \emph{pseudo-BE algebra} is an algebra $(A, \rightarrow, \rightsquigarrow, 1)$ of the type $(2, 2, 0)$ 
such that the following axioms are fulfilled for all $x, y, z\in A$: \\
$(psBE_1)$ $x\rightarrow x=x\rightsquigarrow x=1,$ \\
$(psBE_2)$ $x\rightarrow 1=x\rightsquigarrow 1=1,$ \\
$(psBE_3)$ $1\rightarrow x=1\rightsquigarrow x=x,$ \\
$(psBE_4)$ $x\rightarrow (y\rightsquigarrow z)=y\rightsquigarrow (x\rightarrow z),$ \\
$(psBE_5)$ $x\rightarrow y=1$ iff $x\rightsquigarrow y=1$. 
\end{Def}

A pseudo-BE algebra is said to be \emph{proper} if it is not a BE-algebra. \\
In a pseudo-BE algebra $(A,\ra,\rs,1)$, one can define a binary relation $``\leq"$ by \\
$\hspace*{2cm}$ $x\leq y$ iff $x\ra y=1$ iff $x\rs y=1$, for all $x, y\in A$. \\
If $(A,\ra,\rs,1)$ is a pseudo-BE algebra satisfying $x\ra y=x\rs y$, for all $x, y\in A$, then it is a BE-algebra. \\
If there is an element $0$ of a pseudo-BE algebra $(A, \ra, \rs, 1)$, such that $0\le x$ 
(i.e. $0\rightarrow x=0\rightsquigarrow x=1$), for all $x\in A$, then the pseudo-BE algebra is said to be  
\emph{bounded} and it is denoted by $(A, \ra, \rs, 0, 1)$. 
In a bounded pseudo-BE algebra $(A, \ra, \rs, 0, 1)$ we define two negations: \\
$\hspace*{3cm}$ $x^{-}:= x\ra 0$, $x^{\sim}:= x\rs 0$, \\
for all $x\in A$. \\
Obviously $x^{-\sim}= x\vee_1 0$ and $x^{\sim-}= x\vee_2 0$. \\ 
If $A$ is a bounded pseudo-BE algebra we denote: \\ 
$\hspace*{2cm}$ $\Reg(A)=\{x\in A \mid x^{-\sim}=x^{\sim-}= x\}$, the set of all \emph{regular} elements of $A$, \\
$\hspace*{2cm}$ $\Den(A)=\{x\in A \mid x^{-\sim}=x^{\sim-}= 1\}$, the set of all \emph{dense} elements of $A$. \\ 
If $\Reg(A)=A$, then $A$ is said to be \emph{involutive}. 
If a bounded pseudo-BE algebra $A$ satisfies $x^{-\sim}=x^{\sim-}$ for all $x\in A$, then $A$ is called a \emph{good} pseudo-BE algebra. \\
Obviously, if $A$ is involutive, then $A$ is good and $\Den(A)=\{1\}$.   

\begin{prop} \label{psBE-40} $\rm($\cite{Bor2}$\rm)$ Let $(A,\ra,\rs,1)$ be a pseudo-BE algebra. 
Then the following hold for all $x, y, z\in A$: \\
$(1)$ $x\rightsquigarrow (y\rightarrow z)=y\rightarrow (x\rightsquigarrow z);$ \\
$(2)$ $x\rightarrow (y\rightsquigarrow x)=1$ and $x\rightsquigarrow (y\rightarrow x)=1;$ \\
$(3)$ $x\rightarrow (y\rightarrow x)=1$ and $x\rightsquigarrow (y\rightsquigarrow x)=1;$ \\
$(4)$ $x\rightarrow ((x\rightarrow y)\rightsquigarrow y)=1$ and 
      $x\rightsquigarrow ((x\rightsquigarrow y)\rightarrow y)=1$.  
\end{prop}

We will refer to $(A,\ra,\rs,1)$ by its universe $A$. 

\begin{ex} \label{psBE-50} $\rm($\cite{Bor2}$\rm)$
Consider the set $A=\{a,b,c,1\}$ and the operations $\rightarrow,\rightsquigarrow$ given by the following tables:
\[
\hspace{10mm}
\begin{array}{c|ccccc}
\rightarrow & 1 & a & b & c \\ \hline
1 & 1 & a & b & c \\ 
a & 1 & 1 & 1 & 1 \\ 
b & 1 & a & 1 & c \\ 
c & 1 & b & 1 & 1
\end{array}
\hspace{10mm} 
\begin{array}{c|ccccc}
\rightsquigarrow & 1 & a & b & c \\ \hline
1 & 1 & a & b & c \\ 
a & 1 & 1 & 1 & 1 \\ 
b & 1 & c & 1 & c \\ 
c & 1 & c & 1 & 1
\end{array}
. 
\]
Then $(A,\rightarrow,\rightsquigarrow,1)$ is a pseudo-BE algebra. Moreover it is even a pseudo-BCK algebra. 
\end{ex}

\begin{ex} \label{psBE-50-10} $\rm($\cite{Bor4}$\rm)$
Consider the set $A=\{1,a,b,c,d,e\}$ and the operations $\rightarrow,\rightsquigarrow$ given by the following tables:
\[
\hspace{10mm}
\begin{array}{c|ccccccc}
\rightarrow & 1 & a & b & c & d & e \\ \hline
1 & 1 & a & b & c & d & e \\ 
a & 1 & 1 & c & c & d & 1 \\ 
b & 1 & a & 1 & 1 & 1 & e \\ 
c & 1 & a & 1 & 1 & 1 & e \\ 
d & 1 & a & 1 & 1 & 1 & e \\
e & 1 & a & d & d & d & 1 
\end{array}
\hspace{10mm} 
\begin{array}{c|ccccccc}
\rightsquigarrow & 1 & a & b & c & d & e \\ \hline
1 & 1 & a & b & c & d & e \\ 
a & 1 & 1 & b & c & d & 1 \\ 
b & 1 & a & 1 & 1 & 1 & e \\ 
c & 1 & a & 1 & 1 & 1 & e \\ 
d & 1 & a & 1 & 1 & 1 & e \\
e & 1 & a & c & c & d & 1
\end{array}
. 
\]
Then $(A,\rightarrow,\rightsquigarrow,1)$ is a pseudo-BE algebra. 
Since $b\rightarrow c=1$ and $c\rightarrow b=1$, but $b\ne c$, axiom $(psBCK_6)$ is not satisfied, hence $A$ is 
not a pseudo-BCK algebra.    
\end{ex}

\begin{ex} \label{psBE-50-10-10} $\rm($\cite{Ciu30}$\rm)$
Consider the structure $(A,\ra,\rs,1)$, where the operations $\ra$ and $\rs$ on $A=\{1,a,b,c,d,e\}$ 
are defined as follows:
\[
\begin{array}{c|cccccc}
\ra & 1 & a & b & c & d & e \\ \hline
1 & 1 & a & b & c & d & e \\
a & 1 & 1 & d & 1 & 1 & d \\
b & 1 & c & 1 & 1 & 1 & c \\
c & 1 & a & d & 1 & d & a \\
d & 1 & c & b & c & 1 & b \\
e & 1 & 1 & 1 & 1 & 1 & 1 
\end{array}
\hspace{10mm}
\begin{array}{c|cccccc}
\rs & 1 & a & b & c & d & e \\ \hline
1 & 1 & a & b & c & d & e \\
a & 1 & 1 & c & 1 & 1 & c \\
b & 1 & d & 1 & 1 & 1 & d \\
c & 1 & d & b & 1 & d & b \\
d & 1 & a & c & c & 1 & a \\
e & 1 & 1 & 1 & 1 & 1 & 1
\end{array}
\qquad\quad
\begin{picture}(50,-70)(0,30)
\drawline(0,25)(25,5)(50,25)(0,50)(0,25)(50,50)(50,25)
\drawline(0,50)(25,70)(50,50)
\put(0,25){\circle*{4}}
\put(25,5){\circle*{4}}
\put(50,25){\circle*{4}}
\put(0,50){\circle*{4}}
\put(50,50){\circle*{4}}
\put(25,70){\circle*{4}}
\put(22,-5){$e$}
\put(22,75){$1$}
\put(-11,22){$a$}
\put(-11,47){$c$}
\put(55,22){$b$}
\put(55,47){$d$}
\end{picture}
.
\]
Then $(A,\ra,\rs,e,1)$ is a bounded pseudo-BE algebra.
\end{ex}

\begin{prop} \label{psBE-60} $\rm($\cite{Ciu33}$\rm)$ Any pseudo-BCK algebra is a pseudo-BE algebra.
\end{prop}

\begin{rem} \label{psBE-70}
A revised version of the notion of a \emph{pseudo-equality algebra} has recently been introduced in \cite{Dvu7} as 
an algebra $(A, \wedge, \thicksim, \backsim, 1)$ of the 
type $(2, 2, 2, 0)$ such that the following axioms are fulfilled for all $x, y, z\in A$: \\
$(A_1)$ $(A, \wedge, 1)$ is a meet-semilattice with top element $1,$ \\
$(A_2)$ $x \thicksim x = x \backsim x = 1,$ \\ 
$(A_3)$ $x \thicksim 1 = 1\backsim x = x,$ \\
$(A_4)$ $x \le y \le z$ implies  
$x \thicksim z \le y \thicksim z$, $x \thicksim z \le x \thicksim y$, $z \backsim x \le z \backsim y$ and  
$z \backsim x \le y \backsim x,$ \\ 
$(A_5)$ $x \thicksim y \le (x \wedge z) \thicksim (y \wedge z)$  and 
        $x \backsim y \le (x \wedge z) \backsim (y \wedge z),$ \\  
$(A_6)$ $x \thicksim y \le (z \thicksim x) \backsim (z \thicksim y)$ and 
        $x \backsim y \le (x \backsim z) \thicksim(y \backsim z),$ \\
$(A_7)$ $x \thicksim y \le (x \thicksim z) \thicksim (y \thicksim z)$ and 
        $x \backsim y \le (z \backsim x) \backsim (z \backsim y)$. \\
It was proved in \cite{Dvu7} that the structure $(A, \ra, \rs, 1)$, where: \\
$\hspace*{2cm}$ $x \ra y = (x\wedge y) \thicksim x$ and $x \rs y = x \backsim (x \wedge y)$ \\ 
is a pseudo-BCK algebra. According to Proposition \ref{psBE-60}, $(A, \ra, \rs, 1)$ is a pseudo-BE algebra. 
\end{rem}

\begin{Def} \label{psBE-70-10} A pseudo BE-algebra with $(A)$ condition or a pseudo BE(A)-algebra
for short, is a pseudo BE-algebra $(A, \rightarrow, \rightsquigarrow, 1)$ such that the operations $\rightarrow$, $\rightsquigarrow$ are antitone in the first variable, that is $(A)$ condition is satisfied:\\
$(A)$ if $x, y\in A$ such that $x\le y$, then: \\
$\hspace*{3cm}$ $y\rightarrow z\le x\rightarrow z$ and $y\rightsquigarrow z\le x\rightsquigarrow z,$ \\
for all $z\in A$. 
\end{Def}

\begin{ex} \label{psBE-70-30}
Let $A=\{1,a,b,c,d\}$. Define the operations $\rightarrow$ and $\rightsquigarrow$ on $A$ as follows:
\begin{eqnarray*}
\begin{array}{c|ccccc} \rightarrow & 1 & a & b & c & d\\
 \hline 1 & 1 & a & b & c & d\\
a & 1 & 1 & c & c & 1 \\
b & 1 & d & 1 & 1 & d\\
c & 1 & d & 1 & 1 & d \\
d & 1 & 1 & c & c & 1
\end{array}
\hspace{2cm}
\begin{array}{c|ccccc} \rightsquigarrow & 1 & a & b & c & d \\
\hline 1 & 1 & a & b & c  & d \\
a & 1 & 1 & b & c & 1 \\
b & 1 & d & 1 & 1 & d \\
c & 1 & d & 1 & 1 & d \\
d & 1 & 1 & b & c & 1
\end{array}
\end{eqnarray*}
Then $(A, \rightarrow, \rightsquigarrow, 1)$ is a pseudo-BE(A) algebra. \\
Since $b\rightarrow c=1$ and $c\rightarrow b=1$, but $b\ne c$, axiom $(psBCK_6)$ is not satisfied, hence $A$ is 
not a pseudo-BCK algebra.  
\end{ex}

\begin{prop} \label{psBE-70-40} Let $(A, \rightarrow, \rightsquigarrow, 1)$  be a pseudo BE(A)-algebra and $x, y\in A$
such that $x\le y$. Then $(x\rightarrow z)\rightsquigarrow z \le (y\rightarrow z)\rightsquigarrow z$ and 
$(x\rightsquigarrow z)\rightarrow z \le (y\rightsquigarrow z)\rightarrow z$,
for all $z\in A$. 
\end{prop}
\begin{proof} It is straightforward.
\end {proof}

\begin{Def} \label{psBE-70-50} A pseudo BE-algebra $A$ is said to be \emph{distributive} if it satisfies 
the following condition: \\
$\hspace*{3cm}$ $x\ra(y\rs z)=(x\ra y)\rs (x\ra z)$, \\
for all $x, y, z\in A$. 
\end{Def}

\begin{ex} \label{psBE-70-60} $\rm($\cite{Bor4}$\rm)$  The pseudo-BE algebras from Examples \ref{psBE-50-10} and 
\ref{psBE-70-30} are distributive. 
\end{ex}

A subset $D$ of a pseudo-BE algebra $A$ is called a \emph{deductive system} of $A$ if it satisfies 
the following axioms: \\
$(ds_1)$ $1\in D,$ \\
$(ds_2)$ $x\in D$ and $x\ra y\in D$ imply $y\in D$. \\
A subset $D$ of $A$ is a deductive system if and only if it satisfies $(ds_1)$ and the axiom: \\
$(ds^{'}_2)$ $x\in D$ and $x\rs y\in D$ imply $y\in D$. \\
Denote by ${\mathcal DS}(A)$ the set of all deductive systems of $A$. \\
A deductive system $D$ of $A$ is \emph{proper} if $D\ne A$. \\
A \emph{maximal} deductive system is a proper deductive system such that it is not included in any other proper deductive system. 
Denote by ${\mathcal DS_m}(A)$ the set of all maximal deductive systems of $A$. \\
A deductive system $D$ of a pseudo-BE algebra $A$ is said to be \emph{normal} if it satisfies the condition:\\
$(ds_3)$ for all $x, y \in A$, $x \ra y \in D$ iff $x \rs y \in D$. \\
Denote by ${\mathcal DS_n}(A)$ the set of all normal deductive systems of $A$. \\
For details regarding deductive systems and congruence relations on a pseudo-BE algebra we refer the reader to \cite{Bor2}, \cite{Rez1}. 

\begin{prop} \label{psBE-70-70} $\rm($\cite{Bor4}$\rm)$ If $A$ is a distributive pseudo-BE algebra, then 
${\mathcal DS}(A)={\mathcal DS_n}(A)$. 
\end{prop}

\begin{Def} \label{psBE-80} Let $A$ be a pseudo-BE algebra. Then $P\in {\mathcal DS}(A)$ is called \emph{prime} 
if and only if, for all $D_1, D_2\in {\mathcal DS}(A)$ such that $D_1\cap D_2 \subseteq P$, we have 
$D_1\subseteq P$ or $D_2\subseteq P$. 
The set of all proper prime deductive systems of $A$ is denoted by ${\mathcal DS_p}(A)$ and called the 
\emph{prime spectrum} of $A$.
\end{Def}

\begin{ex} \label{psBE-90} Let $(A, \rightarrow, \rightsquigarrow, 1)$ be the pseudo-BE algebra from 
Example \ref{psBE-50}. Then ${\mathcal DS}(A)=\{\{1\},\{1,b\},A\}$,  
${\mathcal DS_p}(A)={\mathcal DS_m}(A)=\{\{1,b\}\}$.
\end{ex}

Let $A, B$ be two pseudo-BE algebras. A map $f: A\longrightarrow B$ is called a \emph{pseudo-BE homomorphism} 
if $f(x\ra y)=f(x)\ra f(y)$ and $f(x\rs y)=f(x)\rs f(y)$, for all $x, y\in A$. \\
If $B=A$, then $f$ is called a \emph{pseudo-BE endomorphism}. \\
One can easily check that, if $f$ is a pseudo-BE homomorphism, then: \\
$(1)$ $f(1)=1;$ \\
$(2)$ $x\le y$ implies $f(x)\le f(y)$. \\
(We use the same notations for the operations in both pseudo-BE algebras, but the reader must be aware that they are different). \\
Denote $\Ker(f)=\{x\in A \mid f(x)=1\}$. 

\begin{prop} \label{psBE-90-10} $\rm($\cite[Th. 14]{Bor2}$\rm)$ 
Let $f: A\longrightarrow B$ be a \emph{pseudo-BE homomorphism}. Then: \\
$(1)$ if $E\in {\mathcal DS}(B)$, then $f^{-1}(E)\in {\mathcal DS}(A);$ \\
$(2)$ if $f$ s surjective and $D\in {\mathcal DS}(A)$ such that $\Ker(f)\subseteq D$, then 
$f(D)\in {\mathcal DS}(B)$. 
\end{prop}

Let $A$ be a distributive pseudo-BE algebra. 
Given $H\in {\mathcal DS}(A)$, the relation $\Theta_H$ on $A$ defined by $(x,y)\in \Theta_H$ iff 
$x\rightarrow y\in H$ and $y\rightarrow x\in H$ is a congruence on $A$. 
We write $x/H=[x]_{\Theta_H}$ for every $x\in A$ and we have $H=[1]_{\Theta_H}$.  
Then $(A/\Theta_H,\ra,[1]_{\Theta_H})=(A/\Theta_H,\rs,[1]_{\Theta_H})$ is a 
BE-algebra called \emph{quotient pseudo-BE algebra via $H$} and denoted by $A/H$ (see \cite{Rez1}). \\
The function $\pi_H: A \longrightarrow A/H$ defined by $\pi_H(x)=x/H$ for any $x\in A$ is a surjective homomorphism which is called the \emph{canonical projection} from $A$ to $A/H$. \\
One can easily prove that $\Ker(\pi_H)=H$. \\

The proofs of the following two results are straightforward. 

\begin{lemma} \label{psBE-100} In any pseudo-BE algebra $A$ the following hold for all $x, y\in A:$ \\ 
$(1)$ $x\vee_1 1=1\vee_1 x=1$ and $x\vee_2 1=1\vee_2 x=1;$ \\
$(2)$ $x\le y$ implies $x\vee_1 y=x\vee_2 y=y;$ \\
$(3)$ $x\vee_1 x=x\vee_2 x=x;$ \\
$(4)$ $x\vee_1 y$ and $x\vee_2 y$ are uper bounds of $\{x, y\}$. 
\end{lemma}

\begin{lemma} \label{psBE-110} In any pseudo-BE(A) algebra $A$ the following holds for all $x, y\in A:$ \\ 
$\hspace*{2cm}$ $x_1\vee_1 y\le x_2\vee_1 y$ and $x_1\vee_2 y\le x_2\vee_2 y$, \\
whenever $x_1, x_2\in A$, $x_1\le x_2$.  
\end{lemma}

$\vspace*{5mm}$

\section{States, measures and internal states on pseudo-BE algebras}

The Bosbach states and state-morphisms on pseudo-BE algebras were defined and studied in \cite{Rez6}. 
In this section we recall some results on states and state-morphisms on pseudo-BE algebras. 
We generalize to the pseudo-BE algebras the notions of measures, state-measures and measure-morphisms 
investigated in \cite{Ciu25} for the case of pseudo-BCK algebras. It is shown that any measure morphism 
on a pseudo-BE algebra is a measure, and the kernel of a measure is a normal deductive system. 
If $A$ is a bounded pseudo-BE(A) algebra, then we prove that there is a one-to-one correspondence between the set 
of all Bosbach states $s$ on $A$ with $s(0)=0$ and the set of all state-measures on $A$.
The notions of internal states and state-morphism operators on pseudo-BCK algebras are extended to the case of pseudo-BE algebras and some of their properties are investigated. 
We prove that any type II state operator on a pseudo-BE algebra is a state-morphism operator on it. 

\begin{Def} \label{s-psBE-10}
A \emph{Bosbach state} on a pseudo-BE algebra $(A,\ra,\rs,1)$ is a function $s: A\longrightarrow [0, 1]$ 
such that the following axioms hold for all $x, y\in A:$ \\
$(bs_1)$ $s(1)=1;$ \\
$(bs_2)$ $s(x)+s(x\ra y)=s(y)+s(y\ra x);$ \\
$(bs_3)$ $s(x)+s(x\rs y)=s(y)+s(y\rs x)$. 
\end{Def} 

Denote by $\mathcal{BS}(A)$ the set of all states on $A$. 

\begin{prop} \label{s-psBE-20} $\rm($\cite[Lemma 3.4]{Rez6}$\rm)$ 
Let $s\in \mathcal{BS}(A)$. Then the following hold for all $x, y\in A$: \\
$(1)$ $s(y\ra x)=1+s(x)-s(y) =s(y\rs x)$ and $s(x)\le s(y)$ whenever $x\le y;$ \\
$(2)$ $s(x\vee_1 y) = s(y \vee_1 x)$ and $s(x\vee_2 y) = s(y \vee_2 x).$ 
\end{prop}

\begin{prop} \label{s-psBE-20-10} $\rm($\cite[Prop. 3.5]{Rez6}$\rm)$ 
Let $s: A\longrightarrow [0, 1]$ be a function on $A$ such that $s(1)=1$. Then the following are equivalent: \\
$(a)$ $s\in \mathcal{BS}(A);$ \\
$(b)$ $x\leq y$ implies $s(y\ra x)=1+ s(x)- s(y)=s(y\rs x);$ \\
$(c)$ $s(x\ra y)=1+s(y)-s(x\vee_1 y)$ and $s(x\rs y)=1+s(y)-s(x\vee_2 y)$. \\
for all $x, y\in A.$
\end{prop}

\begin{prop} \label{s-psBE-30} $\rm($\cite[Prop. 3.10]{Rez6}$\rm)$ Let $A$ be a pseudo-BE(A) algebra and let 
$s\in \mathcal{BS}(A)$. 
Then the following hold for all $x, y\in A$: \\
$(1)$ $s(x\vee_1 y)=s(x\vee_2 y);$ \\
$(2)$ $s(x\ra y)=s(x\rs y)$. 
\end{prop}

If $(A, \ra, \rs, 0, 1)$ is a bounded pseudo-BE algebra, then we denote 
$\mathcal{BS}_1(A)=\{s\in \mathcal{BS}(A)\mid s(0)=0\}$.

\begin{prop} \label{s-psBE-30-10} $\rm($\cite[Prop. 3.16]{Rez6}$\rm)$ Let $A$ be a bounded pseudo-BE algebra 
and let $s\in \mathcal{BS}_1(A)$. Then the following hold for all $x, y\in A$: \\
$(1)$ $s(x^{-})=s(x^{\sim})=1-s(x);$ \\
$(2)$ $s(x^{-\sim})=s(x^{\sim-})=s(x)$.  
\end{prop}

Let $s\in \mathcal{BS}(A)$ and define $\Ker(s)=\{x\in A \mid s(x)=1\}$,
called the \emph{kernel} of $s$. Then $\Ker(s)\in {\mathcal DS}(A)$ (\cite[Prop. 3.8]{Rez6}). \\
If $A$ is distributive or $A$ is a pseudo-BE(A) algebra, then $\Ker(s)\in {\mathcal DS_n}(A)$ 
(\cite[Th. 9]{Bor4}). 

\begin{theo} \label{s-psBE-40} Let $A$ be a distributive pseudo-BE(A) algebra, $s\in \mathcal{BS}(A)$ and 
$K=\Ker(s)$. Then $(A/K,\ra,1/K)=(A/K,\rs,1/K)$ and $x/K\vee y/K=y/K\vee x/K$, 
for all $x, y\in A$. 
\end{theo} 
\begin{proof}
According to \cite[Th. 3.15]{Rez6}, $(A/K,\ra,\rs,1/K)$ is a pseudo-BE algebra such that: \\ 
$\hspace*{2cm}$ $x/K\vee_1 y/K=y/K\vee_1 x/K$ and $x/K\vee_2 y/K=y/K\vee_2 x/K$, \\ 
for all $x, y\in A$. \\
By \cite[Th. 7]{Bor4}, $x\ra y\le (z\ra x)\rs (z\rs y)$ and $x\rs y\le (z\ra x)\ra (z\ra y)$. \\ 
Taking $z:=1$ we get $x\ra y\le x\rs y$ and $x\rs y\le x\ra y$, that is \\ 
$\hspace*{2cm}$ $(x\ra y)\ra (x\rs y)=1$ and $(x\rs y)\ra (x\ra y)=1$. \\
If follows that \\ 
$\hspace*{2cm}$ $s((x\ra y)\ra (x\rs y))=1$ and $s((x\rs y)\ra (x\ra y))=1$, \\ 
hence $(x\ra y) \Theta_K (x\rs y)$. Similarly $(y\ra x) \Theta_K (y\rs x)$. \\
Thus $(x\ra y)/K=(x\rs y)/K$. \\
We conclude that $x/K\vee y/K=y/K\vee x/K$. 
\end{proof}

\begin{ex} \label{s-psBE-50}
Let $A$ be the pseudo-BE(A) algebra from Example \ref{psBE-70-30}. \\
Define $s^1_{\alpha}, s^2_{\alpha}, s^3_{\alpha,\beta}, s^4: A\longrightarrow [0, 1]$ by: \\
$\hspace*{2cm}$
$s^1_{\alpha}(1)=s^1_{\alpha}(a)=s^1_{\alpha}(d)=1, s^1_{\alpha}(b)=s^1_{\alpha}(c)=\alpha;$ \\ 
$\hspace*{2cm}$
$s^2_{\alpha}(1)=s^2_{\alpha}(b)=s^2_{\alpha}(c)=1, s^2_{\alpha}(a)=s^2_{\alpha}(d)=\alpha;$ \\ 
$\hspace*{2cm}$
$s^3_{\alpha,\beta}(1)=1, s^3_{\alpha,\beta}(a)=s^3_{\alpha,\beta}(d)=\alpha,  
s^3_{\alpha,\beta}(b)=s^3_{\alpha,\beta}(c)=\beta;$ \\ 
$\hspace*{2cm}$
$s^4(1)=s^4(a)=s^4(b)=s^4(c)=s^4(d)=1$. \\ 
Then we have: \\ 
$\hspace*{2cm}$
$\mathcal{BS}(A)=\{s^1_{\alpha} \mid \alpha \in [0, 1)\} \cup
                     \{s^2_{\alpha} \mid \alpha \in [0, 1)\} \cup $                                   
        $\{s^3_{\alpha,\beta} \mid \alpha,\beta\in [0, 1)\} \cup \{s^4\}$. 
\end{ex}

Consider the real interval $[0,1]$ of reals equipped with the \L ukasiewicz implication
$\ra_{\mbox{\tiny \L}}$ defined by
$$x \ra_{\mbox{\tiny \L}} y = \min\{1-x+y,1\}, \ {\rm for \ all} \ x,y \in [0,1].$$

\begin{Def}\label{sm-psBE-10} Let $(A, \ra, \ra, 1)$ be a pseudo BE-algebra. A \emph{state-morphism} on
$A$ is a function $s : A\longrightarrow [0,1]$ such that:\\
$(sm)$ $s(x \ra y) = s(x \rs y)=s(x) \ra_{\mbox{\tiny \L}} s(y),$ for all $x,y \in X.$
\end{Def} 

Denote by $\mathcal{SM}(A)$ the set of all state-morphisms on $A$. 

\begin{prop} \label{sm-psBE-20} $\rm($\cite[Prop. 5.3]{Rez6}$\rm)$ $\mathcal{SM}(A)\subseteq \mathcal{BS}(A)$. 
\end{prop}

\begin{prop} \label{sm-psBE-30} $\rm($\cite[Prop. 5.4]{Rez6}$\rm)$ 
Let $A$ be a pseudo-BE(A) algebra and let $s\in \mathcal{BS}(A)$. 
Then $s\in \mathcal{SM}(A)$ if and only if:
$$s(x\vee_1 y)=\max\{s(x),s(y)\}$$ for all $x,y \in A,$ or equivalently,
$$s(x\vee_2 y)=\max\{s(x),s(y)\}$$ for all $x,y \in A.$
\end{prop}

\begin{ex} \label{sm-psBE-40} 
Consider again the pseudo-BE(A) algebra $A$ from Example \ref{psBE-70-30}. \\
With the notations from Example \ref{s-psBE-50}, we have: \\
$\hspace*{2cm}$
$\mathcal{SM}(A)=\{s^1_{\alpha} \mid \alpha\in {\mathbb R}, \alpha \in [0, 1)\} \cup
                     \{s^2_{\alpha} \mid \alpha\in {\mathbb R}, \alpha \in [0, 1)\} \cup \{s^4\}$. \\ 
Obviously $\mathcal{SM}(A)\subseteq \mathcal{BS}(A)$. \\
Let $s\in \{s^3_{\alpha,\beta} \mid \alpha,\beta\in [0, 1)\}$. 
Then we have $s(a\vee_1 b)=s(a\vee_2 b)=s(1)=1$, while $\max\{s(a),s(b)\}=\max\{\alpha, \beta\}<1$. 
Hence $s$ does not satisfy the conditions from Proposition \ref{sm-psBE-30}. 
\end{ex}

\begin{Def} \label{m-psBE-10} Let $(A,\ra,\rs,1)$ be a pseudo-BE algebra.
A mapping $m:A\longrightarrow [0, \infty)$ such that for all $x,y \in A$, \\
$(1)$ $m(x\ra y)=m(x\rs y)=m(y)-m(x)$ whenever $y\leq x$ is said to be a \emph{measure};\\
$(2)$ if $A$ is bounded, and $m$ is a measure with $m(0)=1$, then $m$ is said to be a \emph{state-measure};\\
$(3)$ if $m(x\ra y)=m(x\rs y)=\max\{0,m(y)-m(x)\}$ is said to be a \emph{measure-morphism};\\
$(4)$ if $A$ is bounded, and $m$ is a measure-morphism with $m(0)=1$, then $m$ is said to be a \emph{state-measure-morphism}.
\end{Def}

Denote by: \\
$\hspace*{2cm}$ $\mathcal{M}(A)$ the set of all measures on $A;$ \\
$\hspace*{2cm}$ $\mathcal{MS}(A)$ the set of all state-measures on $A;$ \\
$\hspace*{2cm}$ $\mathcal{MM}(A)$ the set of all measure-morphisms on $A;$ \\
$\hspace*{2cm}$ $\mathcal{MMS}(A)$ the set of all state-measure-morphisms on $A$. 
 
\begin{prop} \label{m-psBE-20} Let $m\in \mathcal{M}(A)$. For all $x,y \in A$, we have: \\
$(1)$ $m(1)=0;$ \\
$(2)$ $m(x)\geq m(y)$ whenever $x\leq y;$  \\
$(3)$ $m(x\vee_1 y) = m(y\vee_1 x)$ and $m(x\vee_2 y) = m(y\vee_2 x);$ \\
$(4)$ $m(x\vee_1 y) = m(x\vee_2 y);$ \\
$(5)$ $m(x\ra y)=m(x\rs y).$
\end{prop}
\begin{proof} Similarly as \cite[Prop. 4.2]{Ciu25} for the case of pseudo-BCK algebras.
\end{proof}

\begin{prop} \label{m-psBE-20-10} Let $A$ be a pseudo-BE(A) algebra and let $m\in \mathcal{M}(A)$. 
For all $x,y \in A$, we have: \\
$\hspace*{2cm}$ $m(x\vee_1 y\ra y)=m(x\ra y)$ and $m(x\vee_1 y\rs y)=m(x\rs y)$. 
\end{prop}
\begin{proof} 
Since $x\le x\vee_1 y$ we have $x\vee_1 y\ra y \le x\ra y$ and applying Proposition \ref{m-psBE-20}$(2)$ we get 
$m(x\vee_1 y\ra y)\ge m(x\ra y)$. 
On the other hand, $x\ra y\le ((x\ra y)\rs y)\ra y=x\vee_1 y\ra y$. Hence $m(x\ra y)\ge m(x\vee_1 y\ra y)$. 
It follows that $m(x\vee_1 y\ra y)=m(x\ra y)$. \\ 
Similarly $m(x\vee_1 y\rs y)=m(x\rs y)$. 
\end{proof}

\begin{prop} \label{m-psBE-30} Let $A$ be a pseudo-BE algebra. Then:  \\
$(1)$ $y\leq x$ implies $m(x\vee_1 y)=m((x\vee_2 y)=m(x)$ whenever $m\in \mathcal{M}(A);$ \\
$(2)$ if $m\in \mathcal{M}(A)$, then $\Ker_0(m)=\{x\in A \mid m(x)=0\} \in {\mathcal DS}_n(A);$ \\
$(3)$ $\mathcal{MM}(A\subseteq \mathcal{M}(A).$ 
\end{prop}
\begin{proof} Similarly as \cite[Prop. 4.3]{Ciu25} for the case of pseudo-BCK algebras.
\end{proof}

\begin{prop} \label{m-psBE-40} Let $A$ be a bounded pseudo-BE algebra. 
If $s\in \mathcal{BS}_1(A)$, then $m=1-s\in \mathcal{MS}(A)$.
\end{prop}
\begin{proof} 
Let $x, y\in A$, $y\le x$, so $y\ra x=y\rs x=1$. \\
By Proposition \ref{s-psBE-20}$(2)$ we have $s(x\ra y)=s(x\rs y)=1-s(x)+s(y)$. 
It follows that:
$\hspace*{2cm}$ $m(x\ra y)=1-s(x\ra y)=s(x)-s(y)$ \\
$\hspace*{3.8cm}$ $=1-s(y)-(1-s(x)=m(y)-m(x)$. \\ 
Similarly $m(x\rs y)=m(y)-m(x)$. We also have $m(0)=1-s(0)=1$. \\
Hence $m\in \mathcal{MS}(A)$.
\end{proof}

\begin{prop} \label{m-psBE-40-10} Let $A$ be a bounded pseudo-BE(A) algebra. 
If $m\in \mathcal{MS}(A)$, then $s=1-m\in \mathcal{BS}_1(A)$.
\end{prop}
\begin{proof} Let $m\in \mathcal{MS}(A)$ and let $x, y\in A$. \\
Since $y\le x\vee_1 y$, we get: \\
$\hspace*{2cm}$ $m(x\vee_1 y\ra y)=m(x\vee_1 y\rs y)=m(y)-m(x\vee_1 y)$. \\
According to Proposition \ref{m-psBE-20-10} we have: \\ 
$\hspace*{2cm}$ $m(x\vee_1 y\ra y)=m(x\ra y)$ and $m(x\vee_1 y\rs y)=m(x\rs y)$. \\
Hence $m(x\ra y)=m(y)-m(x\vee_1 y)$ and $m(x\rs y)=m(y)-m(x\vee_2 y)$. \\
Similarly $m(y\ra x)=m(x)-m(y\vee_1 x)$ and $m(y\rs x)=m(x)-m(y\vee_2 x)$. \\ 
Applying Proposition \ref{m-psBE-20}$(3)$ we get: \\
$\hspace*{2cm}$ $m(x)+m(x\ra y)=m(y)+m(y\ra x)$, \\
$\hspace*{2cm}$ $m(x)+m(x\rs y)=m(y)+m(y\rs x)$. \\
It follows that: \\
$\hspace*{2cm}$ $s(x)+s(x\ra y)=s(y)+s(y\ra x)$, \\
$\hspace*{2cm}$ $s(x)+s(x\rs y)=s(y)+s(y\rs x)$. \\
Moreover $s(0)=1-m(0)=0$ and $s(1)=1-m(1)=1$. \\
We conclude that $s\in \mathcal{BS}_1(A)$.
\end{proof} 

\begin{theo} \label{m-psBE-40-20} Let $A$ be a bounded pseudo-BE(A) algebra.  
There is a one-to-one correspondence between $\mathcal{BS}_1(A)$ and $\mathcal{MS}(A)$.
\end{theo} 
\begin{proof}
It follows by Propositions \ref{m-psBE-40} and \ref{m-psBE-40-10}. 
\end{proof}

\begin{ex} \label{m-psBE-50}
Let $A$ be the pseudo-BE(A) algebra from Example \ref{psBE-70-30}. \\
Define $m^1_{\alpha}, m^2_{\alpha}, m^3_{\alpha,\beta}, m^4: A\longrightarrow [0, \infty)$ by: \\
$\hspace*{2cm}$
$m^1_{\alpha}(1)=m^1_{\alpha}(a)=m^1_{\alpha}(d)=0, m^1_{\alpha}(b)=m^1_{\alpha}(c)=\alpha;$ \\ 
$\hspace*{2cm}$
$m^2_{\alpha}(1)=m^2_{\alpha}(b)=m^2_{\alpha}(c)=0, m^2_{\alpha}(a)=m^2_{\alpha}(d)=\alpha;$ \\ 
$\hspace*{2cm}$
$m^3_{\alpha,\beta}(1)=0, m^3_{\alpha,\beta}(a)=m^3_{\alpha,\beta}(d)=\alpha,  
m^3_{\alpha,\beta}(b)=m^3_{\alpha,\beta}(c)=\beta;$ \\ 
$\hspace*{2cm}$
$m^4(1)=m^4(a)=m^4(b)=m^4(c)=m^4(d)=0$. \\ 
Then we have: \\ 
$\hspace*{2cm}$
$\mathcal{M}(A)=\{m^1_{\alpha} \mid \alpha\in {\mathbb R}, \alpha > 0\} \cup 
                     \{m^2_{\alpha} \mid \alpha\in {\mathbb R}, \alpha > 0\} \cup $ \\     
$\hspace*{3.5cm}$                                 
        $\{m^3_{\alpha,\beta} \mid \alpha,\beta\in {\mathbb R}, \alpha,\beta > 0\} \cup \{m^4\};$ \\ 
$\hspace*{2cm}$
$\mathcal{MM}(A)=\{m^1_{\alpha} \mid \alpha\in {\mathbb R}, \alpha > 0\} \cup 
                     \{m^2_{\alpha} \mid \alpha\in {\mathbb R}, \alpha > 0\} \cup \{m^4\}$. \\     
We can see that $\mathcal{M}(A)\subset \mathcal{MM}(A)$, but $\mathcal{M}(A)\ne \mathcal{MM}(A)$. 
\end{ex}

\begin{Def} \label{is-psBE-10} Let $(A, \rightarrow, \rightsquigarrow, 1)$ be a pseudo-BE algebra   
and $\mu:A \longrightarrow A$ be a unary operator on $A$. For all $x, y\in A$ consider the following axioms:\\
$(is_1)$ $\mu(x)\le \mu(y)$, whenever $x\le y,$ \\
$(is_2)$ $\mu(x\ra y)=\mu(x\vee_1 y)\ra \mu(y)$ and $\mu(x\rs y)=\mu(x\vee_2 y)\rs \mu(y),$ \\
$(is^{'}_2)$ $\mu(x\ra y)=\mu(y\vee_1 x)\ra \mu(y)$ and $\mu(x\rs y)=\mu(y\vee_2 x)\rs \mu(y),$ \\ 
$(is_3)$ $\mu(\mu(x)\ra \mu(y))=\mu(x)\ra \mu(y)$ and $\mu(\mu(x)\rs \mu(y))=\mu(x)\rs \mu(y)$. \\
Then: \\
$(i)$ $\mu$ is called an \emph{internal state of type I} or a \emph{state operator of type I} or 
a \emph{type I state operator} if it satisfies axioms $(is_1)$, $(is_2)$, $(is_3);$ \\    
$(ii)$ $\mu$ is called an \emph{internal state of type II} or a \emph{state operator of type II} or a 
\emph{type II state operator} if it satisfies axioms $(is_1)$, $(is^{'}_2)$, $(is_3)$. \\
The structure $(A, \rightarrow, \rightsquigarrow, \mu, 1)$ ($(A,\mu)$, for short) is called a 
\emph{state pseudo-BE algebra of type I (type II) state pseudo-BE algebra}, respectively.
\end{Def}

Denote $\mathcal{IS}^{(I)}(A)$ and $\mathcal{IS}^{(II)}(A)$ the set of all internal states of 
type I and II on a pseudo-BE algebra $A$, respectively. \\
For $\mu \in \mathcal{IS}^{(I)}(A)$ or $\mu \in \mathcal{IS}^{(II)}(A)$, 
$\Ker(\mu)=\{x\in A \mid \mu(x)=1\}$  is called the \emph{kernel} of $\mu$. 

\begin{prop} \label{is-psBE-20} Let $(A, \ra, \rs, \mu, 1)$ be a type I or a type II state pseudo-BE(A) algebra. 
Then the following hold:\\ 
$(1)$ $\mu(1)=1;$ \\
$(2)$ $\mu(\mu(x))=\mu(x)$, for all $x\in A;$ \\
$(3)$ $\mu(x\rightarrow y)\le \mu(x)\rightarrow \mu(y)$ and 
      $\mu(x\rightsquigarrow y)\le \mu(x)\rightsquigarrow \mu(y)$, for all $x, y\in A;$ \\
$(4)$ $\Ker(\mu)\in {\mathcal DS}(A);$ \\
$(5)$ $\Img(\mu)$ is a subalgebra of $A;$ \\ 
$(6)$ $\Img(\mu)=\{x\in A \mid x=\mu(x)\};$ \\
$(7)$ $\Ker(\mu)\cap \Img(\mu)=\{1\}$. 
\end{prop}
\begin{proof}
Similarly as \cite[Prop. 5.5]{Ciu7}, based on (A) condition.
\end{proof}

\begin{prop} \label{is-psBE-30} Let $(A, \ra, \rs, \mu, 1)$ be a type II pseudo-BE(A) algebra. 
Then the following hold:\\ 
$(1)$ $y\le x$ implies $\mu(x\ra y)=\mu(x)\ra \mu(y)$ and $\mu(x\rs y)=\mu(x)\rs \mu(y);$ \\
$(2)$ $x\ra y\in \Ker(\mu)$ iff $y\vee_1 x\ra y\in \Ker(\mu);$ \\            
$(3)$ $x\rs y\in \Ker(\mu)$ iff $y\vee_2 x\rs y\in \Ker(\mu).$     
\end{prop}
\begin{proof} 
$(1)$ It follows applying Lemma \ref{psBE-100}$(2)$. \\
$(2)$ Suppose $x\ra y\in \Ker(\mu)$, that is $\mu(x\ra y)=1$. 
Since by Lemma \ref{psBE-100}$(4)$, $y\le y\vee_1 x$, applying $(1)$ we get   
$\mu(y\vee_1 x \ra y)=\mu(y\vee_1 x)\ra \mu(y)$, hence \\ 
$\hspace*{2cm}$ $1=\mu(x\ra y)=\mu(y\vee_1 x)\ra \mu(y)=\mu(y\vee_1 x\ra y)$. \\ 
Thus $y\vee_1 x\ra y\in \Ker(\mu)$. \\ 
Conversely, suppose $y\vee_1 x\ra y\in \Ker(\mu)$. \\
From $x\le y\vee_1 x$, by (A) condition, we have $y\vee_1 x \ra y\le x\ra y$. \\ 
Since $\Ker(\mu)\in {\mathcal DS}(A)$, we get $x\ra y\in \Ker(\mu)$. \\
$(3)$ Similarly as $(2)$.  
\end{proof}

\begin{ex} \label{is-psBE-50} Consider the pseudo-BE algebra $(A,\ra,\rs,1)$ from Example \ref{psBE-70-30} 
and the maps $\mu_i:A\longrightarrow A$, $i=1,2,\cdots,10$, given in the table below:
\[
\begin{array}{c|cccccc}
 x & 1 & a & b & c & d  \\ \hline
\mu_1(x) & 1 & a & a & a & a \\
\mu_2(x) & 1 & b & b & b & b \\
\mu_3(x) & 1 & c & c & c & c \\
\mu_4(x) & 1 & d & d & d & d \\
\mu_5(x) & 1 & d & c & c & d \\
\mu_6(x) & 1 & 1 & b & b & 1 \\
\mu_7(x) & 1 & 1 & c & c & 1 \\
\mu_8(x) & 1 & a & 1 & 1 & a \\
\mu_9(x) & 1 & d & 1 & 1 & d \\
\mu_{10}(x) & 1 & 1 & 1 & 1 & 1 
\end{array}
.   
\]
Then: $\mathcal{IS}^{(I)}(A)=\mathcal{IS}^{(II)}(A)=\{\mu_1, \mu_2,\cdots,\mu_{10}\}$. 
\end{ex}

\begin{Def} \label{smo-ps-10} Let $(A, \ra,\rs,1)$ be a pseudo-BE algebra. 
A homomorphism $\mu:A\longrightarrow A$ is called a \emph{state-morphism operator} on $A$ if $\mu^2=\mu$, where $\mu^2=\mu\circ \mu$. The pair $(A, \mu)$ is called a \emph{state-morphism pseudo-BE algebra}.
\end{Def}

Denote $\mathcal{SMO}(A)$ the set of all state-morphism operators on a pseudo-BE algebra $A$.

\begin{ex} \label{smo-psBE-50} Consider the pseudo-BE algebra $(A,\ra,\rs,1)$ from Example \ref{psBE-70-30} 
and the maps $\mu_i:A\longrightarrow A$, $i=1,2,\cdots,9$, given in the table below:
\[
\begin{array}{c|cccccc}
 x & 1 & a & b & c & d  \\ \hline
\mu_1(x) & 1 & d & c & c & d \\
\mu_2(x) & 1 & 1 & b & b & 1 \\
\mu_3(x) & 1 & 1 & c & c & 1 \\
\mu_4(x) & 1 & a & 1 & 1 & a \\
\mu_5(x) & 1 & d & 1 & 1 & d \\
\mu_6(x) & 1 & 1 & 1 & 1 & 1 \\
\mu_7(x) & 1 & a & b & c & d \\
\mu_8(x) & 1 & a & c & c & d \\
\mu_9(x) & 1 & d & b & c & d 
\end{array}
.   
\]
Then: $\mathcal{SMO}(A)=\{\mu_1, \mu_2,\cdots,\mu_{9}\}$. 
\end{ex}

\begin{rem} \label{smo-psBE-60} Let $(A,\ra,\rs,1)$ be a pseudo-BCK algebra. 
According to \cite[Remark 6.4]{Ciu7}, $\mathcal{SMO}(A)\subseteq \mathcal{IS}^{(I)}(A)$. 
As we can see in Examples \ref{is-psBE-50} and \ref{smo-psBE-50}, this result is not valid in the case 
of pseudo-BE algebras. 
\end{rem}

\begin{prop} \label{smo-psBE-70} Let $(A,\ra,\rs,1)$ be a linearly ordered pseudo-BE algebra. \\
Then $\mathcal{IS}^{(II)}(A)\subseteq  \mathcal{SMO}(A)$. 
\end{prop}
\begin{proof} Let $\mu\in \mathcal{IS}^{(II)}(A)$. 
According to Proposition \ref{is-psBE-20}$(2)$, $\mu^2=\mu$. \\
Let $x, y\in A$ such that $x\le y$, so $x\ra y=x\rs y=1$ and $\mu(x)\le \mu(y)$. \\
Applying Proposition \ref{is-psBE-20}$(1)$,$(3)$ we get: \\
$\hspace*{2cm}$ $1=\mu(1)=\mu(x\ra y)\le \mu(x)\ra \mu(y)=1$, \\
hence $\mu(x\ra y)=\mu(x)\ra \mu(y)$ and similarly $\mu(x\rs y)=\mu(x)\rs \mu(y)$. \\
If $y\le x$, then by Lemma \ref{psBE-100}$(2)$ we have: \\
$\hspace*{2cm}$ $\mu(x\ra y)=\mu(y\vee_1 x)\ra \mu(y)=\mu(x)\ra \mu(y)$ and \\
$\hspace*{2cm}$ $\mu(x\rs y)=\mu(y\vee_2 x)\rs \mu(y)=\mu(x)\rs \mu(y)$. \\ 
It follows that $\mu\in \mathcal{SMO}(A)$, hence $\mathcal{IS}^{(II)}(A)\subseteq  \mathcal{SMO}(A)$. 
\end{proof}

$\vspace*{5mm}$

\section{On commutative pseudo-BE algebras}

Commutative pseudo-BE algebras were defined and investigated in \cite{Ciu33}. 
It was proved that the class of commutative pseudo-BE algebras is equivalent to the class of commutative 
pseudo-BCK algebras. Based on this result, all results holding for commutative pseudo-BCK algebras also hold for commutative pseudo-BE algebras. 
For example, any finite commutative pseudo-BE algebra is a BE-algebra, and any commutative pseudo-BE algebra is a 
join-semilattice. Moreover, if a commutative pseudo-BE algebra is a meet-semilattice, then it is a distributive lattice.
In this section we recall some properties and results on commutative pseudo-BE algebras which will be used in the 
next sections. We generalize to the case of commutative pseudo-BE algebras the axiom system given in \cite{Rez2} for commutative BE-algebras. A characterization of commutative pseudo-BE algebras is also given. 
We prove that in the case of a commutative pseudo-BE algebra the two types of internal states coincide.
Moreover, if a commutative pseudo-BE algebra is linearly ordered, then any internal state is also a 
state-morphism operator. 

\begin{Def} \label{comm-psBE-10} $\rm($\cite{Ciu33}$\rm)$ A pseudo-BE algebra $(A,\rightarrow,\rightsquigarrow,1)$ is said to be \emph{commutative} if it satisfies the following conditions, for all $x, y\in A$: \\ 
$\hspace*{3cm}$ $x\vee_1 y=y\vee_1 x$ and $x\vee_2 y=y\vee_2 x$. 
\end{Def}

Obviously any bounded commutative pseudo-BE algebra is involutive. 

\begin{ex} \label{comm-psBE-10-10} The pseudo-BE algebra $(A,\rightarrow,\rightsquigarrow,1)$ from 
Example \ref{psBE-50} is not commutative, since $a\vee_2 c=c\neq b=c\vee_2 a$. 
\end{ex}

\begin{ex} \label{comm-psBE-10-20} 
Let $(G, \vee,\wedge, \cdot, ^{-1}, e)$ be an $\ell$-group. On the negative cone $G^{-}=\{g\in G \mid g\le e\}$ we define the operations $x\rightarrow y=y\cdot (x\vee y)^{-1}$, $x\rightsquigarrow y=(x\vee y)^{-1}\cdot y$. 
Then $(G^{-}, \rightarrow, \rightsquigarrow, e)$ is a commutative pseudo-BE algebra. 
\end{ex}

\begin{ex} \label{comm-psBE-10-30} Let $(A,\ra,\rs,1)$ be a distributive pseudo-BE(A) algebra and let 
$s\in \mathcal{BS}(A)$. If $K=\Ker(s)$, then according to Theorem \ref{s-psBE-40}, 
$(A/K,\ra,1/K)=(A/K,\rs,1/K)$ is a commutative BE-algebra.
\end{ex} 

\begin{theo} \label{comm-psBE-20} $\rm($\cite{Ciu33}$\rm)$ Any commutative pseudo-BE algebra is a pseudo-BCK algebra. 
\end{theo}

\begin{theo} \label{comm-psBE-30} $\rm($\cite{Ciu33}$\rm)$ The class of commutative pseudo-BE algebras is 
equivalent to the class of commutative pseudo-BCK algebras. 
\end{theo}

\begin{rem} \label{comm-psBE-45} $\rm($\cite{Ciu33}$\rm)$ As a consequence of Theorem \ref{comm-psBE-30}, 
all results holding for commutative pseudo-BCK algebras also hold for commutative pseudo-BE algebras. 
We recall some of these results:\\ 
$(1)$ A pseudo-BE algebra $(A,\ra,\rs,1)$ is commutative if and only if $y\ra x=y\rs x=1$ implies 
$x\vee_1 y=x\vee_2 y=x$, for all $x, y\in A$ (\cite[Lemma 4.1.4]{Kuhr6}). \\
$(2)$ There are no proper finite commutative pseudo-BE algebras (\cite[Corollary 4.1.6]{Kuhr6}). 
In the other words, every finite commutative pseudo-BE algebra is a commutative BE-algebra. \\ 
$(3)$ If $(A,\ra,\rs,1)$ is a commutative pseudo-BE algebra, then  
$x\vee_1 y=x\vee_2 y$ (\cite[Corollary 1.2]{Ciu2}). \\
$(4)$ If $(A,\ra,\rs,1)$ is a commutative pseudo-BE algebra,  
then $(A,\vee,\ra,\rs,1)$ is a join-semilattice, where $x\vee y=x\vee_1 y=x\vee_2 y$ (\cite{Ior1}). \\ 
Note that generally, the underlying poset $(A,\le)$ of a pseudo-BE algebra $(A,\ra,\rs,1)$ need not be a join-semilattice (see Example \ref{psBE-50-10-10}). \\  
$(5)$ For a commutative pseudo-BE algebra $(A,\ra,\rs,1)$, if $(A,\le)$ is a meet-semilattice, then it is a distributive lattice (\cite[Corollary 4.1.9]{Kuhr6}). \\
$(6)$ The prime spectrum ${\mathcal DS_p}(A)$  of a commutative pseudo-BE algebra $(A,\ra,\rs,1)$ is equipped 
with the \emph{spectral topology} (see (4) and \cite{Kuhr6}). 
\end{rem}

\begin{rem} \label{comm-psBE-90} $\rm($\cite{Ciu33}$\rm)$
Pseudo-MV algebras were introduced by G. Georgescu and A. Iorgulescu in \cite{Geo2}, and independently by 
J. Rach{\accent23u}nek in \cite{Rac2}. It was proved that bounded commutative pseudo-BCK algebras are categorically 
isomorphic with pseudo-MV algebras (see \cite{Geo15}, \cite{Ior1}). 
By Theorem \ref{comm-psBE-30}, it follows that bounded commutative pseudo-BE algebras are also categorically 
isomorphic with pseudo-MV algebras. 
\end{rem}

\begin{prop} \label{comm-psBE-100} In any commutative pseudo-BE algebra $(A,\ra,\rs,1)$ 
the following hold for all $x, y\in A:$ \\ 
$(1)$ $x\ra y=y\vee_1 x\ra y$ and $x\rs y=y\vee_2 x\rs y;$ \\
$(2)$ $x\vee_1 y=(x\vee_1 y)\vee_1 x$ and $x\vee_2 y=(x\vee_2 y)\vee_2 x;$ \\
$(3)$ $x\le y$ implies $y\vee_1 x=y\vee_2 x=y$.  
\end{prop}
\begin{proof} It follows by Theorem \ref{comm-psBE-20} and \cite[Th. 3.7]{Ciu7}.  
\end{proof}

\begin{prop} \label{comm-psBE-110} If the pseudo-BE algebra $A$ is commutative, then  $\mathcal{IS}^{(I)}(A)=\mathcal{IS}^{(II)}(A)$, but the converse is not always true.
\end{prop}
\begin{proof}
If $A$ is commutative, then obviously $\mathcal{IS}^{(I)}(A)=\mathcal{IS}^{(II)}(A)$. \\
In the case of pseudo-BE(A) algebra $A$ from Example \ref{is-psBE-50}, we have $\mathcal{IS}^{(I)}(A)=\mathcal{IS}^{(II)}(A)$, but $A$ is not commutative. 
\end{proof}

\begin{rem} \label{comm-psBE-120} If $A$ is a pseudo-BCK algebra such that   $\mathcal{IS}^{(I)}(A)=\mathcal{IS}^{(II)}(A)$, then $A$ is commutative (\cite[Prop. 5.4]{Ciu7}). 
As we can see in Proposition \ref{comm-psBE-110}, this result is not valid in the case of pseudo-BE algebras.  
\end{rem}

\begin{prop} \label{comm-psBE-130} If $(A,\ra,\rs,1)$ is a linearly ordered commutative pseudo-BE algebra, 
then $\mathcal{IS}^{(I)}(A)\subseteq \mathcal{SMO}(A)$ and $\mathcal{IS}^{(II)}(A)\subseteq \mathcal{SMO}(A)$. 
\end{prop} 
\begin{proof} 
By Proposition \ref{smo-psBE-70}, we have $\mathcal{IS}^{(II)}(A)\subseteq \mathcal{SMO}(A)$. \\
Let $\mu\in \mathcal{IS}^{(I)}(A)$. By Proposition \ref{is-psBE-20}$(2)$, $\mu^2=\mu$. \\ 
Consider $x, y\in A$ such that $x\le y$. Similarly as in Proposition \ref{smo-psBE-70}, we get \\
$\hspace*{2cm}$ $\mu(x\ra y)=\mu(x)\ra \mu(y)$ and $\mu(x\rs y)=\mu(x)\rs \mu(y)$. \\
If $y\le x$, since $A$ is commutative, applying Lemma \ref{psBE-100}$(2)$ we have: \\ 
$\hspace*{2cm}$ $\mu(x\ra y)=\mu(x\vee_1 y)\ra \mu(y)=\mu(y\vee_1 x)\ra \mu(y)=\mu(x)\ra \mu(y)$ and \\
$\hspace*{2cm}$ $\mu(x\rs y)=\mu(x\vee_2 y)\ra \mu(y)=\mu(y\vee_2 x)\rs \mu(y)=\mu(x)\rs \mu(y)$. \\ 
It follows that $\mu\in \mathcal{SMO}(A)$, hence $\mathcal{IS}^{(I)}(A)\subseteq  \mathcal{SMO}(A)$. 
\end{proof}

\begin{theo} \label{comm-psBE-140} An algebra $(A,\ra,\rs,1)$ of the type $(2,2,0)$ is a commutative 
pseudo-BE algebra if and only if it satisfies the following identities for all $x, y, z\in A:$ \\
$(P_1)$ $1\ra x=1\rs x=x;$ \\
$(P_2)$ $x\ra 1=x\rs 1=1;$ \\
$(P_3)$ $(x\ra z)\rs (y\ra z)=(z\ra x)\rs (y\ra x)$ and \\ 
$\hspace*{0.7cm}$ $(x\rs z)\ra (y\rs z)=(z\rs x)\ra (y\rs x);$ \\
$(P_4)$ $x\ra (y\rs z)=y\rs (x\ra z);$ \\
$(P_5)$ $x\ra y=1$ iff $x\rs y=1$. 
\end{theo} 
\begin{proof}
Assume that $A$ is a commutative pseudo-BE algebra. 
Conditions $(P_1)$, $(P_2)$, $(P_4)$ and $(P_5)$ are in fact axioms $(psBE_3)$, $(psBE_2)$, $(psBE_4)$ and 
$(psBE_5)$, respectively. \\ 
Using $(psBE_4)$ and applying the commutativity we get: \\ 
$\hspace*{2cm}$ $(x\ra z)\rs (y\ra z)=y\ra ((x\ra z)\rs z)=y\ra ((z\ra x)\rs x)$ \\ 
$\hspace*{5.3cm}$ $=(z\ra x)\rs (y\ra x)$, \\
$\hspace*{2cm}$ $(x\rs z)\ra (y\rs z)=y\rs ((x\rs z)\ra z)=y\ra ((z\rs x)\ra x)$ \\ 
$\hspace*{5.3cm}$ $=(z\rs x)\ra (y\rs x)$, \\
that is $(P_3)$. \\
Conversely, let $(A,\rightarrow,\rightsquigarrow,1)$ be an algebra satisfying conditions $(P_1)-(P_5)$. \\ 
Axioms $(psBE_2)$, $(psBE_3)$, $(psBE_4)$ and $(psBE_5)$ are in fact conditions $(P_2)$, $(P_1)$, $(P_4)$ and 
$(P_5)$, respectively. \\
By $(P_1)$, $(P_3)$ and  $(P_2)$ we have: \\
$\hspace*{2cm}$ $x\ra x=(1\rs x)\ra (1\rs x)=(x\rs 1)\ra (x\rs 1=1\ra 1=1$, \\
$\hspace*{2cm}$ $x\rs x=(1\ra x)\rs (1\ra x)=(x\ra 1)\rs (x\ra 1=1\rs 1=1$, \\
that is $(psBE_1)$. \\ 
It follows that $A$ is a pseudo-BE algebra. \\
Using axioms $(P_1)$ and $(P_3)$ we get:\\ 
$\hspace*{2cm}$ $(x\ra y)\rs y=(x\ra y)\rs (1\ra y)=(y\ra x)\rs (1\ra x)=(y\ra x)\rs x$, \\
$\hspace*{2cm}$ $(x\rs y)\ra y=(x\rs y)\ra (1\rs y)=(y\rs x)\ra (1\rs x)=(y\rs x)\ra x$, \\
hence $A$ is commutative. 
\end{proof}

\begin{theo} \label{comm-psBE-150} An algebra $(A,\ra,\rs,1)$ of the type $(2,2,0)$ is a 
commutative pseudo-BE algebra if and only if it satisfies the following identities for all $x, y, z\in A:$ \\
$(Q_1)$ $(x\ra 1)\rs y=(x\rs 1)\ra y=y;$ \\
$(Q_2)$ $(x\ra z)\rs (y\ra z)=(z\ra x)\rs (y\ra x)$ and \\ 
$\hspace*{0.8cm}$ $(x\rs z)\ra (y\rs z)=(z\rs x)\ra (y\rs x);$ \\
$(Q_3)$ $x\ra (y\rs z)=y\rs (x\ra z);$ \\
$(Q_4)$ $x\ra y=1$ iff $x\rs y=1$. 
\end{theo} 
\begin{proof}
Assume that $A$ is a commutative pseudo-BE algebra. 
Condition $(Q_1)$ follows from $(psBE_2)$ and $(psBE_3)$, while conditions $(Q_3)$ and $(Q_4)$ 
are axioms $(psBE_4)$ and $(psBE_5)$, respectively. \\
Using $(psBE_4)$ and based on the commutativity we get: \\ 
$\hspace*{2cm}$ $(x\ra z)\rs (y\ra z)=y\ra ((x\ra z)\rs z)=y\ra ((z\ra x)\rs x)$ \\ 
$\hspace*{5.3cm}$ $=(z\ra x)\rs (y\ra x)$, \\
$\hspace*{2cm}$ $(x\rs z)\ra (y\rs z)=y\rs ((x\rs z)\ra z)=y\ra ((z\rs x)\ra x)$ \\ 
$\hspace*{5.3cm}$ $=(z\rs x)\ra (y\rs x)$, \\
hence $(Q_2)$ is satisfied. \\
Conversely, let $(A,\ra,\rs,1)$ be an algebra satisfying conditions $(Q_1)-(Q_4)$. \\ 
Axioms $(psBE_4)$ and $(psBE_5)$ are in fact conditions $(Q_3)$ and $(Q_4)$, respectively. \\
Applying $(Q_1)$ twice we have: \\
$\hspace*{2cm}$ $1\ra x=((1\ra 1)\rs 1)\ra x=x;$ \\
$\hspace*{2cm}$ $1\rs x=((1\rs 1)\ra 1)\rs x=x$, \\ 
thus $(psBE_3)$ is satisfied. \\
Using $(psBE_3)$, $(Q_1)$ and $(Q_2)$ we get: \\
$\hspace*{2cm}$ $x\ra x=(1\rs x)\ra (1\rs x)=(1\rs(1\ra x))\ra (1\rs (1\ra x))$ \\ 
$\hspace*{3.1cm}$ $=((1\ra x)\rs 1)\ra (1\rs 1)=1\ra 1=1$, \\
$\hspace*{2cm}$ $x\rs x=(1\ra x)\rs (1\ra x)=(1\ra(1\rs x))\rs (1\ra (1\rs x))$ \\ 
$\hspace*{3.1cm}$ $=((1\rs x)\ra 1)\rs (1\ra 1)=1\rs 1=1$, \\
that is $(psBE_1)$. \\
By $(Q_1)$ and $(psBE_1)$ we have: \\
$\hspace*{2cm}$ $x\ra 1=(x\ra 1)\rs (x\ra 1)=1$, \\
$\hspace*{2cm}$ $x\rs 1=(x\rs 1)\ra (x\rs 1)=1$, \\
hence $(psBE_2)$ is satisfied. \\
Thus $A$ is a pseudo-BE algebra. \\
Applying $(Q_1)$ and $(Q_2)$ we have: \\
$\hspace*{2cm}$ $(x\ra y)\rs y=(x\ra y)\rs ((y\rs 1)\ra y)$ \\
$\hspace*{4.2cm}$ $=(y\ra x)\rs ((y\rs 1)\ra x)=(y\ra x)\rs x$, \\                   
$\hspace*{2cm}$ $(x\rs y)\ra y=(x\rs y)\ra ((y\ra 1)\rs y)$ \\
$\hspace*{4.2cm}$ $=(y\rs x)\ra ((y\ra 1)\rs x)=(y\rs x)\ra x$. \\                   
It follows that $A$ is commutative. 
\end{proof}

$\vspace*{5mm}$

\section{Fantastic deductive systems of pseudo-BE algebras}

In this section we define the fantastic deductive systems of pseudo-BE algebras and we investigate their 
properties. 
It is proved that, if a pseudo-BE(A) algebra $A$ is commutative, then all deductive systems of $A$ are fantastic. 
We show that the kernel of a Bosbach state (state-morphism, measure, type II state operator) is a fantastic 
deductive system. \\
We define the notion of an involutive deductive system of a bounded pseudo-BE algebra and we prove that every 
fantastic deductive system on a bounded pseudo-BE algebra is an involutive deductive system. 
It is also proved that the kernel of a Bosbach state on a bounded pseudo-BE algebra is an involutive deductive 
system. 

\begin{Def} \label{comm-ds-70} Let $(A,\rightarrow,\rightsquigarrow,1)$ be a pseudo-BE algebra and 
let $D\in {\mathcal DS}(A)$. Then $D$ is called \emph{fantastic} if it satisfies the following conditions 
for all $x, y\in A$: \\ 
$(cds_1)$ $y\rightarrow x\in D$ implies $x\vee_1 y\rightarrow x\in D;$ \\
$(cds_2)$ $y\rightsquigarrow x\in D$ implies $x\vee_2 y\rightsquigarrow x\in D$. 
\end{Def}

We will denote by ${\mathcal DS}_f(A)$ the set of all fantastic deductive systems of a pseudo-BE algebra $A$. 

\begin{prop} \label{comm-ds-80} Let $(A,\rightarrow,\rightsquigarrow,1)$  be a pseudo-BE algebra and 
let $D\subseteq A$. Then $D\in {\mathcal DS}_f(A)$ if and only if it satisfies the following conditions 
for all $x, y, z\in A$: \\
$(1)$ $1\in D;$ \\
$(2)$ $z\rightarrow (y\rightarrow x)\in D$ and $z\in D$ implies $x\vee_1 y\rightarrow x\in D;$ \\
$(3)$ $z\rightsquigarrow (y\rightsquigarrow x)\in D$ and $z\in D$ implies $x\vee_2 y\rightsquigarrow x\in D$.
\end{prop}
\begin{proof}
Consider $D\in {\mathcal DS}_f(A)$. Since $1\in D$, condition $(1)$ is satisfied. \\
Let $x, y, z\in A$ such that $z\rightarrow (y\rightarrow x), z\in D$.  
Since $D\in {\mathcal DS}(A)$, we have $y\rightarrow x\in D$, hence 
$x\vee_1 y\ra x\in D$, that is condition $(2)$. \\ 
Similarly from $z\rightsquigarrow (y\rightsquigarrow x)\in D$ and $z\in D$ we get 
$x\vee_2 y\rs x\in D$, that is condition $(3)$.  \\
Conversely, let $D\subseteq A$ satisfying conditions $(1)$, $(2)$ and $(3)$. Obviously $1\in D$. \\ 
Let $x, y\in D$ such that $x\rightarrow y, x\in D$. \\
Since $x\rightarrow (1\rightarrow y)=x\rightarrow y\in D$, using $(2)$ we have $y=y\vee_1 1\ra y\in D$. \\
It follows that $D\in {\mathcal DS}(A)$. \\
Let $x, y\in A$ such that $y\rightarrow x\in D$. Since $1\rightarrow (y\rightarrow x)\in D$ and $1\in D$, 
by $(2)$ we get $x\vee_1 y\ra x\in D$. \\ 
Similarly from $y\rightsquigarrow x\in D$ we get $x\vee_2 y\rs x\in D$. \\  
We conclude that $D\in {\mathcal DS}_f(A)$. 
\end{proof}

\begin{exs} \label{comm-ds-90} 
$(1)$ Let $A=\{1,a,b,c,d,e\}$ be the pseudo-BE(A) algebra from Example \ref{psBE-50-10}. \\ 
One can see that: \\
$\hspace*{2cm}$ ${\mathcal DS}(A)=\{\{1\},\{1,e\},\{1,a,e\},\{1,b,c,d\},\{1,b,c,d,e\},A\}$, \\
$\hspace*{2cm}$ ${\mathcal DS}_n(A)=\{\{1\},\{\{1,e\},\{1,a,e\},\{1,b,c,d\},\{1,b,c,d,e\},A\}$, \\  
$\hspace*{2cm}$ ${\mathcal DS}_f(A)=\{\{1,e\},\{1,a,e\},\{1,b,c,d,e\},A\}$. \\
$(2)$ Let $A=\{1,a,b,c,d\}$ be the pseudo-BE(A) algebra from Example \ref{psBE-70-30}. Then: \\ 
$\hspace*{2cm}$ ${\mathcal DS}(A)={\mathcal DS}_n(A)={\mathcal DS}_f(A)=\{\{1\},\{1,a,d\},\{1,b,c\},A\}$. 
\end{exs}

\begin{prop} \label{comm-ds-110} Let $(A,\ra,\rs,1)$ be a pseudo-BE(A) algebra and 
$D\in {\mathcal DS}_f(A)$, $E\in {\mathcal DS}(A)$ such that $D\subseteq E$. Then $E\in {\mathcal DS}_f(A)$. 
\end{prop}
\begin{proof}
Consider $x, y\in A$ such that $u=y\ra x\in E$. It follows that \\
$\hspace*{2cm}$ $y\rightarrow (u\rs x)=y\ra ((y\ra x)\rs x)=1\in D$. \\
Since D is fantastic, we have $(u\rs x)\vee_1 y\ra (u\rs x)\in D$. \\
From $D\subseteq E$ we get $(u\rs x)\vee_1 y\ra (u\rs x)\in E$. \\
Applying $(psBE_4)$, it follows that $u\rs ((u\rs x)\vee_1 y\ra x)\in E$. \\
Since $u\in E$, we get $(u\rs x)\vee_1 y\ra x\in E$. \\
From $x\le u\rs x$, by condition (A), we have $(u\rs x)\ra y\le x\ra y$, and \\
$(x\ra y)\rs y\le ((u\rs x)\ra y\le x)\rs y$, that is $x\vee_1 y\le (u\rs x)\vee_1 y$. \\
Finally, applying again condition (A), $(u\rs x)\vee_1 y\ra x\le x\vee_1 y\ra x$. 
Hence $x\vee_1 y\ra x\in E$. \\
Similarly from $y\rs x\in E$, we get $x\vee_2 y\rs x\in E$. \\
We conclude that $E\in {\mathcal DS}_f(A)$.
\end{proof}

\begin{cor} \label{comm-ds-120} Let $A$ be a pseudo-BE(A) algebra. Then $\{1\}\in {\mathcal DS}_f(A)$
if and only if ${\mathcal DS}(A)={\mathcal DS}_f(A)$.
\end{cor}

\begin{theo} \label{comm-ds-120-10} If $A$ is a commutative pseudo-BE algebra, then 
${\mathcal DS}(A)={\mathcal DS}_f(A)$.
\end{theo}
\begin{proof} 
Let $D\in {\mathcal DS}(A)$ and let $x, y\in A$ such that $y\ra x\in D$. 
By Proposition \ref{psBE-40}$(4)$, $y\ra x\le ((y\ra x)\rs x)\ra x$, hence 
$((y\ra x)\rs x)\ra x=y\vee_1 x\ra x\in D$. \\ 
Since $A$ is commutative, we get $x\vee_1 y\ra x\in D$. \\
Similarly $y\rs x\in D$ implies $x\vee_2 y\rs x\in D$, hence $D\in {\mathcal DS}_f(A)$. \\
We conclude that ${\mathcal DS}(A)\subseteq {\mathcal DS}_f(A)$, that is ${\mathcal DS}(A)={\mathcal DS}_f(A)$.
\end{proof}

\begin{rem} \label{comm-ds-130} Let $(A,\rightarrow,\rightsquigarrow,1)$ be a pseudo-BCK algebra. 
According to \cite[Th. 4.7, Cor. 4.8]{Ciu7}, the following are equivalent: \\ 
$(1)$ $A$ is commutative; \\
$(2)$ $\{1\}\in {\mathcal DS}_f(A);$ \\  
$(3)$ ${\mathcal DS}(A)={\mathcal DS}_f(A)$. \\ 
As we can see in Example \ref{comm-ds-90}$(2)$, this result is not valid in the case of pseudo-BE algebras. 
Indeed, $A$ satisfies $(2)$ and $(3)$, but it is not commutative, since $a\vee_1 d=d\neq a=d\vee_1 a$. \\
However, $(A,\vee,\rightarrow,\rightsquigarrow,1)$ is a join-semilattice with 
$x\vee y=x\vee_1 y=x\vee_2 y$, for all $x, y\in A$.
\end{rem}

\begin{prop} \label{comm-ds-130-10} Let $f: A\longrightarrow B$ be a pseudo-BE homomorphism. 
If $D\in {\mathcal DS}_f(B)$ then $f^{-1}(D)\in {\mathcal DS}_f(A)$. 
\end{prop}
\begin{proof}
Consider $D\in {\mathcal DS}_f(B)$ and let $x, y\in A$ such that $y\ra x\in f^{-1}(D)$, that is 
$f(y\ra x)\in D$, so $f(y)\ra f(x)\in D$. Since $D\in {\mathcal DS}_f(B)$, we have  
$f(x)\vee_1 f(y)\ra f(x)\in D$. \\
It follows that $f(x\vee_1 y\ra x)\in D$, hence $x\vee_1 y\ra x\in f^{-1}(D)$. \\
Similarly $y\rs x\in f^{-1}(D)$ implies $x\vee_2 y\rs x\in f^{-1}(D)$, thus $f^{-1}(D)\in {\mathcal DS}_f(A)$. 
\end{proof}

\begin{prop} \label{comm-ds-140} If $s\in \mathcal{BS}(A)$, then $\Ker(s)\in {\mathcal DS}_f(A)$. 
\end{prop}
\begin{proof}
Let $x, y\in A$ such that $y\ra x \in \Ker(s)$, that is $s(y\ra x)=1$. \\ 
It follows that $s(x\ra y)=s(y)+s(y\ra x)-s(x)=s(y)+1-s(x)$. \\
Moreover, since $y\le x\ra y$ and $x\le x\vee_1 y$, we have: \\
$\hspace*{2cm}$ $s(y\rs (x\ra y))=s(1)=1$ and $s(x\ra x\vee_1 y)=s(1)=1$. \\
Then we get: \\
$\hspace*{2cm}$ $s(x\ra y)+s((x\ra y)\rs y)=s(y)+s(y\rs (x\ra y))=s(y)+1$ \\
$\hspace*{2cm}$ $s(x\vee_1 y)=s((x\ra y)\rs y)=1+s(y)-s(x\ra y)=s(x)$. \\
$\hspace*{2cm}$ $s(x\vee_1 y\ra x)=s(x)+s(x\ra x\vee_1 y))-s(x\vee_1 y)=1$. \\
It follows that $x\vee_1 y\ra x\in \Ker(s)$. \\ 
Similarly from $y\rs x \in \Ker(s)$ we get $x\vee_1 y\rs x\in \Ker(s)$. \\
We conclude that $\Ker(s)\in {\mathcal DS}_f(A)$. 
\end{proof}

\begin{ex} \label{comm-ds-150} Consider the Bosbach states from Example \ref{s-psBE-50}. Then we have: \\
$\hspace*{2cm}$ $\Ker(s^1_{\alpha})=\{1, a, d\}$, $\Ker(s^2_{\alpha})=\{1, b, c\}$, 
                $\Ker(s^3_{\alpha, \beta})=\{1\}$, $\Ker(s^4)=A$. \\ 
We can see that 
    $\{\Ker(s^1_{\alpha}), \Ker(s^2_{\alpha}), \Ker(s^3_{\alpha, \beta}), \Ker(s^4)\}={\mathcal DS}_f(A)$.  
\end{ex}

\begin{prop} \label{comm-ds-160} If $s\in \mathcal{SM}(A)$, then $\Ker(s)\in {\mathcal DS}_f(A)$. 
\end{prop}
\begin{proof} 
Let $s\in \mathcal{SM}(A)$. 
By Proposition \ref{sm-psBE-20}, $s\in \mathcal{BS}(A)$ and applying Proposition \ref{comm-ds-140}, it 
follows that $\Ker(s)\in {\mathcal DS}_f(A)$. 
\end{proof}

\begin{ex} \label{comm-ds-170} Consider the state-morphisms from Example \ref{sm-psBE-40}. \\
Obviously $\mathcal{SM}(A)\subseteq \mathcal{BS}(A)$ and 
$\{\Ker(s^1_{\alpha}), \Ker(s^2_{\alpha}), \Ker(s^4)\}\subseteq {\mathcal DS}_f(A)$.       
\end{ex}

\begin{prop} \label{comm-ds-180} If $m\in \mathcal{M}(A)$, then $\Ker_0(m)\in {\mathcal DS}_f(A)$. 
\end{prop}
\begin{proof} 
Consider $x, y\in A$ such that $y\ra x\in \Ker_0(m)$, that is $m(y\ra x)=0$. \\
According to Proposition \ref{m-psBE-20}, $m(x\vee_1 y)=m(y\vee_1 x)$.  \\
Since $x\le x\vee_1 y$ and $x\le y\ra x$, we have: \\
$\hspace*{2cm}$ $m(x\vee_1 y\ra x)=m(x)-m(x\vee_1 y)=m(x)-m(y\vee_1 x)$ \\
$\hspace*{4.6cm}$ $=m(x)-m((y\ra x)\rs x)$ \\
$\hspace*{4.6cm}$ $=m(x)-m(x)+m(y\ra x)=0$, \\ 
hence $x\vee_1 y\ra x\in \Ker_0(m)$. \\
Similarly $y\rs x\in \Ker_0(m)$ implies $x\vee_2 y\rs x\in \Ker_0(m)$. \\
It follows that $\Ker_0(m)\in {\mathcal DS}_f(A)$.
\end{proof}

\begin{ex} \label{comm-ds-190} Consider the measures from Example \ref{m-psBE-50}. Then we have: \\
$\hspace*{2cm}$ $\Ker_0(m^1_{\alpha})=\{1, a, d\}$, $\Ker_0(m^2_{\alpha})=\{1, b, c\}$, 
                $\Ker_0(m^3_{\alpha, \beta})=\{1\}$, $\Ker_0(m^4)=A$. \\ 
We can see that 
    $\{\Ker_0(m^1_{\alpha}), \Ker_0(m^2_{\alpha}), \Ker_0(m^3_{\alpha, \beta}), \Ker_0(m^4)\}={\mathcal DS}_f(A)$.   
\end{ex}

\begin{prop} \label{comm-ds-200} Let $(A, \ra, \rs, \mu, 1)$ be a type II pseudo-BE(A) algebra. \\
Then $\Ker(\mu)\in {\mathcal DS}_f(A)$.   
\end{prop}
\begin{proof} It follows by Proposition \ref{is-psBE-30}$(2),(3)$. 
\end{proof}

\begin{ex} \label{comm-psBE-210} Consider the internal states from Example \ref{is-psBE-50}. Then we have: \\
$\hspace*{2cm}$ $\Ker(\mu_1)=\Ker(\mu_2)=\Ker(\mu_3)=\Ker(\mu_4)=\Ker(\mu_5)=\{1\}$, \\ 
$\hspace*{2cm}$ $\Ker(\mu_6)=\Ker(\mu_7)=\{1,a,d\}$, \\                 
$\hspace*{2cm}$ $\Ker(\mu_8)=\Ker(\mu_9)=\{1,b,c\}$, \\ 
$\hspace*{2cm}$ $\Ker(\mu_{10})=A$. \\ 
We can see that $\{\Ker(\mu_i) \mid i=1,2,\cdots,10\}={\mathcal DS}_f(A)$.  
\end{ex}

In what follows by $A$ we will denote a bounded pseudo-BE algebra.   

\begin{Def} \label{comm-ds-220} Let $D\in {\mathcal DS}(A)$. Then $D$ is called an \emph{involutive deductive system} 
of $A$ if $x^{-\sim}\ra x, x^{\sim-}\rs x \in D$, for all $x\in A$.  
\end{Def}

We will denote by ${\mathcal DS}_i(A)$ the set of all involutive deductive systems of $A$. 

\begin{rems} \label{comm-ds-230}
$(1)$ $\Den(A)\subseteq D$, for any $D\in {\mathcal DS}_i(A)$. \\
$(2)$ if $D\in {\mathcal DS}_i(A)$, then $x\in D$ iff $x^{-\sim}\in D$ iff $x^{\sim-}\in D$. \\
$(3)$ if $A$ is involutive then ${\mathcal DS}_i(A)={\mathcal DS}(A)$. 
\end{rems}

\begin{prop} \label{comm-ds-240} If $s\in \mathcal{BS}_1(A)$, then $\Ker(s)\in {\mathcal DS}_i(A)$. 
\end{prop}
\begin{proof}
Consider $s\in \mathcal{BS}_1(A)$ and let $x\in A$. 
Taking $y:=0$ in Proposition \ref{psBE-40}$(4)$ we get $x\le x^{-\sim}$ and $x\le x^{\sim-}$, that is 
$x\ra x^{-\sim}=1$ and $x\rs x^{\sim-}=1$. \\ 
From $(bs_1)-(bs_3)$, applying Proposition \ref{s-psBE-30-10} we get $s(x^{-\sim}\ra x)=1$ and $s(x^{\sim-}\rs x)=1$. \\
Hence $x^{-\sim}\ra x, x^{\sim-}\rs x \in \Ker(s)$, that is $\Ker(s)\in {\mathcal DS}_i(A)$. 
\end{proof}

\begin{prop} \label{comm-ds-250} ${\mathcal DS}_f(A)\subseteq {\mathcal DS}_i(A)$.  
\end{prop}
\begin{proof}
Let $D\in {\mathcal DS}_f(A)$ and $x\in A$. \\
Since $0\ra x=1\in D$, we have $x\vee_1 0\ra x\in D$, so $x^{-\sim}\ra x \in D$. \\
Similarly from $0\rs x=1\in D$, we get $x\vee_2 0\rs x\in D$, that is $x^{\sim-}\rs x \in D$. \\
Hence $D\in {\mathcal DS}_i(A)$ and we conclude that ${\mathcal DS}_f(A)\subseteq {\mathcal DS}_i(A)$. 
\end{proof}


$\vspace*{5mm}$

\section{Valuations on pseudo-BE algebras} 

In this section the notions of pseudo-valuation and commutative pseudo-valuation on pseudo-BE algebras are defined and investigated. Given a pseudo-BE algebra $A$, it is proved that the kernel of a commutative pseudo-valuation on $A$ is a fantastic deductive system of $A$. If moreover $A$ is commutative, then we prove 
that any pseudo-valuation on $A$ is commutative. 
Characterizations of pseudo-valuations and commutative pseudo-valuations are given.
We study the relationships between the pseudo-valuations of homomorphic and isomorphic pseudo-BE algebras and 
the relationships between their kernels. \\
In what follows by $A$ we will denote a pseudo-BE algebra, unless otherwise is stated.   

\begin{Def} \label{va-psbe-10} A real-valued function $\varphi:A\longrightarrow {\mathbb R}$ is called a 
\emph{pseudo-valuation} on $A$ if it satisfies the following conditions: \\
$(pv_1)$ $\varphi(1)=0;$ \\
$(pv_2)$ $\varphi(y)-\varphi(x)\le \min\{\varphi(x\ra y), \varphi(x\rs y)\}$  
for all $x, y\in A$. \\ 
A pseudo-valuation $\varphi$ is said to be a \emph{valuation} if it satisfies the condition: \\
$(pv_3)$ $v(x)=0$ implies $x=1$ for all $x\in A$. 
\end{Def}

Denote $\mathcal{PV}(A)$ the set of all pseudo-valuations on $A$. 

\begin{prop} \label{va-psbe-20} If $\varphi\in \mathcal{PV}(A)$, then the following hold for all $x, y, z\in A$: \\
$(1)$ $\varphi(x)\ge \varphi(y)$, whenever $x\le y$ ($\varphi$ is order reversing); \\
$(2)$ $\varphi(x)\ge 0;$ \\
$(3)$ if $z\ra (y\rs x)=1$ or $z\rs (y\ra x)=1$, then $\varphi(x)\le \varphi(y)+ \varphi(z)$. 
\end{prop}
\begin{proof}
For $x\le y$ we have $\varphi(y)-\varphi(x)\le \min\{\varphi(x\ra y), \varphi(x\rs y)\}= 
\min \{\varphi(1),\varphi(1)\}=0$, that is $\varphi(x)\ge \varphi(y)$. \\
$(2)$ Since $x\le 1$, by $(1)$ we get $\varphi(x)\ge \varphi(1)=0$. \\
$(3)$ By $(pv_2)$ we have \\ 
$\hspace*{2cm}$ $\varphi(y\rs x)-\varphi(z)\le \varphi(z\ra (y\rs x))=\varphi(1)=0$, \\ 
so $\varphi(y\rs x)\le \varphi(z)$. \\
Applying again $(pv_2)$ we get \\ 
$\hspace*{2cm}$ $\varphi(x)-\varphi(y)\le \varphi(y\rs x)\le \varphi(z)$, \\ 
that is $\varphi(x)\le \varphi(y)+ \varphi(z)$. \\
Similarly for the case $z\rs (y\ra x)=1$. 
\end{proof}

\begin{theo} \label{va-psbe-20-10} Let $\varphi:A\longrightarrow {\mathbb R}$ satisfying $(pv_1)$.  
Then $\varphi\in \mathcal{PV}(A)$ if and only if $\varphi$ satisfies the following conditions: \\
$(pv_4)$ $\varphi(x\ra z)\le \varphi(x\ra (y\rs z)) + \varphi(y)$ \\
$(pv_5)$ $\varphi(x\rs z)\le \varphi(x\rs (y\ra z)) + \varphi(y)$, \\
for all $x, y, z\in A$.
\end{theo} 
\begin{proof}
Let $\varphi\in \mathcal{PV}(A)$ and $x, y, z\in A$. 
Applying the definition of a pseudo-valuation and $(psBE_4)$ we have: \\
$\hspace*{2cm}$ $\varphi(x\ra z)-\varphi(y)\le \varphi(y\rs (x\ra z))=\varphi(x\ra (y\rs z))$ \\
$\hspace*{2cm}$ $\varphi(x\rs z)-\varphi(y)\le \varphi(y\ra (x\rs z))=\varphi(x\rs (y\ra z))$, \\ 
that is $(pv_4)$ and $(pv_5)$ are satisfied. \\
Conversely, taking $x:=1$,$z:=y$ and $y:=x$ in $(pv_4)$ and $(pv_5)$ we get: \\
$\hspace*{2cm}$ $\varphi(y)\le \varphi(x\rs y)+\varphi(x)$ and $\varphi(y)\le \varphi(x\ra x)+\varphi(x)$, \\ respectively. It follows that: \\  
$\hspace*{2cm}$ $\varphi(y)-\varphi(x)\le \min\{\varphi(x\rs y),\varphi(x\ra y)\}$, \\ 
that is $(pv_2)$. We conclude that $\varphi\in \mathcal{PV}(A)$. 
\end{proof}

\begin{ex} \label{va-psbe-30}(see \cite{Bus3})
Let $D\in \mathcal{DS}(A)$ and the map $\varphi:A\longrightarrow {\mathbb R}$ defined by 
\begin{equation*}
\varphi(x)= \left\{
\begin{array}{cc}
	0,\:\: {\rm if} \: x \in D\\
	a,\:\: {\rm if} \: x \notin D,
\end{array}
\right.
\end{equation*}
where $a\in {\mathbb R}$, $a\ge 0$. Then $\varphi\in \mathcal{PV}(A)$. 
\end{ex}

\begin{ex} \label{va-psbe-30-05} Let $A$ be the pseudo-BE(A) algebra from Example \ref{psBE-70-30}. \\
Define $\varphi_{\alpha,\beta}: A\longrightarrow {\mathbb R}$ by: $\varphi_{\alpha,\beta}(1)=0$, 
$\varphi_{\alpha,\beta}(a)=v_{\alpha,\beta}(d)=\alpha$, $\varphi_{\alpha,\beta}(b)=\varphi_{\alpha,\beta}(c)=\beta$. \\ 
Then $\mathcal{PV}(A)=\{\varphi_{\alpha,\beta} \mid \alpha, \beta\in {\mathbb R}, \alpha, \beta \ge 0\}$. 
\end{ex}

\begin{Def} \label{va-psbe-30-10} A real-valued function $\varphi:A\longrightarrow {\mathbb R}$ is called a 
\emph{weak pseudo-valuation} on $A$ if it satisfies the following condition: \\
$(pv_6)$ $\max\{\varphi(x\ra y), \varphi(x\rs y)\} \le \varphi(x)+\varphi(y)$, \\
for all $x, y\in A$. 
\end{Def}

Denote $\mathcal{PV}^w(A)$ the set of all weak pseudo-valuations on $A$.

\begin{rem} \label{va-psbe-30-20} If $\varphi\in \mathcal{PV}^w(A)$, then $\varphi(x)\ge 0$ for all $x\in A$. \\ 
Indeed, for any $x\in A$ we have $\varphi(1)=\varphi(x\ra 1)\le \varphi(x)+\varphi(1)$, hence $\varphi(x)\ge 0$. 
\end{rem}

\begin{prop} \label{va-psbe-30-30} $\mathcal{PV}(A)\subseteq \mathcal{PV}^w(A)$. 
\end{prop} 
\begin{proof}
Let $\varphi\in \mathcal{PV}(A)$ and $x, y\in A$. Using $(psBE_4)$, $(psBE_1)$ and $(psBE_2)$ we have: \\
$\hspace*{2cm}$ $x\ra (y\rs (x\ra y))=x\ra (x\ra (y\rs y))=x\ra (x\ra 1)=x\ra 1=1$, \\
$\hspace*{2cm}$ $x\ra (y\ra (x\rs y))=x\ra (x\rs (y\ra y))=x\ra (x\rs 1)=x\ra 1=1$. \\
By $(pv_1)$ and $(pv_2)$ we get: \\
$\hspace*{2cm}$ $0=\varphi(1)=\varphi(x\ra (y\rs (x\ra y)))$ \\ 
$\hspace*{3.5cm}$ $\ge \varphi(y\rs (x\ra y))-\varphi(x)$ \\ 
$\hspace*{3.5cm}$ $\ge \varphi(x\ra y)-\varphi(y)-\varphi(x)$ and \\
$\hspace*{2cm}$ $0=\varphi(1)=\varphi(x\ra (y\ra (x\rs y)))$ \\ 
$\hspace*{3.5cm}$ $\ge \varphi(y\ra (x\rs y))-\varphi(x)$ \\ 
$\hspace*{3.5cm}$ $\ge \varphi(x\rs y)-\varphi(y)-\varphi(x)$. \\
Hence $\varphi(x\ra y)\le \varphi(x)+\varphi(y)$ and $\varphi(x\rs y)\le \varphi(x)+\varphi(y)$. \\
It follows that $\max\{\varphi(x\ra y), \varphi(x\rs y)\} \le \varphi(x)+\varphi(y)$, 
that is $(pv_6)$ is satisfied. \\ 
Thus $\varphi\in \mathcal{PV}^w(A)$, and we conclude that $\mathcal{PV}(A)\subseteq \mathcal{PV}^w(A)$. 
\end{proof}

\begin{ex} \label{va-psbe-30-40} Let $A$ be the pseudo-BE(A) algebra from Example \ref{psBE-70-30}. \\
Define $\varphi_{\alpha_1,\alpha_2,\alpha_3,\alpha_4,\alpha_5}: A\longrightarrow {\mathbb R}$ by: \\
$\hspace*{3cm}$ $\varphi_{\alpha_1,\alpha_2,\alpha_3,\alpha_4,\alpha_5}(1)=\alpha_1$ \\
$\hspace*{3cm}$ $\varphi_{\alpha_1,\alpha_2,\alpha_3,\alpha_4,\alpha_5}(a)=\alpha_2$ \\
$\hspace*{3cm}$ $\varphi_{\alpha_1,\alpha_2,\alpha_3,\alpha_4,\alpha_5}(b)=\alpha_3$ \\
$\hspace*{3cm}$ $\varphi_{\alpha_1,\alpha_2,\alpha_3,\alpha_4,\alpha_5}(c)=\alpha_4$ \\
$\hspace*{3cm}$ $\varphi_{\alpha_1,\alpha_2,\alpha_3,\alpha_4,\alpha_5}(d)=\alpha_5$. \\
Then $\mathcal{PV}^w(A)$ is the set of all maps $\varphi_{\alpha_1,\alpha_2,\alpha_3,\alpha_4,\alpha_5}$  
satisfying the conditions: \\
$\hspace*{3cm}$ $\alpha_i \in {\mathbb R}, \alpha_1\ge 0$, $\alpha_i\ge \frac{1}{2}\alpha_1$, for $i=2,3,4,5;$ \\
$\hspace*{3cm}$ $\alpha_4\le \alpha_3$ or $\alpha_3<\alpha_4\le \min \{\alpha_2+\alpha_3,\alpha_3+\alpha_5\};$ \\
$\hspace*{3cm}$ $\alpha_5\le \min \{\alpha_2+\alpha_3,\alpha_2+\alpha_4\}$.  \\
For example, we can see that $\varphi_{0,1,3,4,2}\in \mathcal{PV}^w(A)$, but 
$\varphi_{0,1,3,4,2}\notin \mathcal{PV}(A)$. 
\end{ex}

For $\varphi\in \mathcal{PV}(A)$, denote $\Ker({\varphi})=\{x\in A\mid \varphi(x)=0\}$, called the 
\emph{kernel} of $\varphi$. 

\begin{rem} \label{va-psbe-30-50} For any $\varphi\in \mathcal{PV}(A)$, $\Ker({\varphi})\in \mathcal{DS}(A)$. \\
Indeed, by $(pv_1)$, $1\in \Ker({\varphi})$. \\
Let $x, y\in A$, such that $x, x\ra y\in \Ker({\varphi})$, that is $\varphi(x)=\varphi(x\ra y)=0$. \\
Using $(pv_2)$, $\varphi(y)-\varphi(x)\le \varphi(x\ra y)$, thus $\varphi(y)=0$. \\ 
It follows that $\Ker({\varphi})\in \mathcal{DS}(A)$.
\end{rem}

\begin{Def} \label{va-psbe-40} A pseudo-valuation $\varphi$ on $A$ is said to be \emph{commutative} if it satisfies the following conditions for all $x, y\in A$: \\
$(cpv_1)$ $\varphi(x\vee_1 y\ra x)\le \varphi(y\ra x);$ \\
$(cpv_2)$ $\varphi(x\vee_2 y\rs x)\le \varphi(y\rs x)$. 
\end{Def}

Denote $\mathcal{PV}^{c}(A)$ the set of all commutative pseudo-valuations on $A$. 

\begin{theo} \label{va-psbe-50} A pseudo-valuation $\varphi$ on $A$ is commutative if and only if it satisfies the following conditions for all $x, y, z\in A:$ \\
$(cpv_3)$ $\varphi(x\vee_1 y\ra x)\le \varphi(z\ra (y\ra x))+\varphi(z);$ \\
$(cpv_4)$ $\varphi(x\vee_2 y\rs x)\le \varphi(z\rs (y\rs x))+\varphi(z)$. 
\end{theo}
\begin{proof}
Let $\varphi \in \mathcal{PV}^{c}(A)$, that is $\varphi$ satisfies conditions $(cpv_1)$ and $(cpv_2)$. \\
By $(cpv_1)$ and $(pv_2)$ we have: \\ 
$\hspace*{2cm}$ $\varphi(x\vee_1 y\ra x)-\varphi(z)\le \varphi(y\ra x)-\varphi(z)\le \varphi(z\ra (y\ra x))$, \\
that is $(cpv_3)$. \\
Similarly from $(cpv_2)$ and $(pv_2)$ we get: \\
$\hspace*{2cm}$ $\varphi(x\vee_1 y\rs x)-\varphi(z)\le \varphi(y\rs x)-\varphi(z)\le \varphi(z\rs (y\rs x))$, \\
thus $(cpv_4)$ is verified. \\
Conversely, let $\varphi \in \mathcal{PV}(A)$ satisfying conditions $(cpv_3)$ and $(cpv_4)$. \\
Taking $z=1$ we get $(cpv_1)$ and $(cpv_2)$, hence $\varphi \in \mathcal{PV}^{c}(A)$. 
\end{proof}

\begin{prop} \label{va-psbe-60} If $\varphi\in \mathcal{PV}^{c}(A)$, then $\Ker({\varphi})\in \mathcal{DS}_f(A)$.  
\end{prop}
\begin{proof}
Let $x, y\in A$ such that $y\ra x\in \Ker({\varphi})$, that is $\varphi(y\ra x)=0$. \\ 
By $(cpv_1)$, $\varphi(x\vee_1 y\ra x)\le \varphi(y\ra x)=0$, 
hence $\varphi(x\vee_1 y\ra x)=0$, so $x\vee_1 y\ra x\in \Ker({\varphi})$. \\
Similary, if $y\rs x\in \Ker({\varphi})$, applying $(cpv_2)$ we get $x\vee_2 y\rs x\in \Ker({\varphi})$. \\
Thus $\Ker({\varphi})\in \mathcal{DS}_f(A)$.
\end{proof}

\begin{prop} \label{va-psbe-70} If $A$ is a commutative pseudo-BE algebra, then 
$\mathcal{PV}^{c}(A)=\mathcal{PV}(A)$.
\end{prop}
\begin{proof} Let $\varphi\in \mathcal{PV}(A)$ and $x, y\in A$. 
According to Theorem \ref{comm-psBE-20}, $A$ is a commutative pseudo-BCK algebra.
Applying Proposition \ref{psBE-20}$(6)$ we have: \\
$\hspace*{2cm}$ $y\ra x=y\vee_1 x\ra x=x\vee_1 y\ra x$, \\ 
$\hspace*{2cm}$ $y\rs x=y\vee_2 x\rs x=x\vee_2 y\rs x$. \\
It follows that $\varphi(x\vee_1 y\ra x)=\varphi(y\ra x)$ and $\varphi(x\vee_2 y\rs x)=\varphi(y\rs x)$. \\   
Hence $(cpv_1)$ and $(cpv_2)$ are satisfied, thus $\varphi\in \mathcal{PV}^{c}(A)$. \\ 
We conclude that $\mathcal{PV}^{c}(A)=\mathcal{PV}(A)$.
\end{proof}

\begin{rem} \label{va-psbe-80} If $\mathcal{PV}^{c}(A)=\mathcal{PV}(A)$, then $A$ need not be commutative. \\
Indeed, consider the pseudo-BE(A) algebra from Example \ref{psBE-70-30}. \\
Using the notations from Example \ref{va-psbe-30-05}, we have \\
$\hspace*{2cm}$ 
$\mathcal{PV}^c(A)=\mathcal{PV}(A)=\{\varphi_{\alpha,\beta} \mid \alpha, \beta\in {\mathbb R}, \alpha,\beta \ge 0\}$, \\ 
but $A$ is not commutative($a\vee_1 d=d \ne a=d\vee_1 a$).
\end{rem}

\begin{theo} \label{va-psbe-90}
Let $f: A\longrightarrow B$ be a pseudo-BE homomorphism and let $\varphi\in \mathcal{PV}(B)$.  
Then there exists a unique $\psi\in \mathcal{PV}(A)$ such that the following diagram is commutative. 
Moreover: \\ 
$(i)$ $\Ker(\psi)=f^{-1}(\Ker(\varphi);$ \\
$(ii)$ if $\Ker({\varphi})\in \mathcal{DS}_f(B)$, then $\Ker({\psi})\in \mathcal{DS}_f(A)$.  
\begin{center}
\begin{picture}(105,100)(0,0)
\put(0,80){$A$}
\put(8,84){\vector(1,0){72}}
\put(82,80){$B$}
\put(40,90){$f$}
\put(42,0){${\mathbb R}$}
\put(5,77){\vector(1,-2){36}}
\put(10,40){$\psi$}
\put(86,77){\vector(-1,-2){36}}
\put(73,40){$\varphi$}
\end{picture}
\end{center}
\end{theo} 
\begin{proof}
Let $\varphi\in \mathcal{PV}(B)$ and define $\psi:=\varphi\circ f$. We get: \\ 
$\hspace*{2cm}$ $\psi(1)=(\varphi\circ f)(1)=\varphi(f(1))=\varphi(1)=0$. \\ 
For all $x, y\in A$ we have: \\
$\hspace*{2cm}$ $\psi(y)-\psi(x)=(\varphi\circ f)(y)-(\varphi\circ f)(y)=
\varphi(f(y))-\varphi (f(x))$ \\
$\hspace*{4.1cm}$
$\le \varphi(f(x)\ra f(y))=\varphi(f(x\ra y))=\psi(x\ra y)$. \\
Similarly $\psi(y)-\psi(x)\le \psi(x \rs y)$. \\
Hence $\psi(y)-\psi(x)\le \min \{\psi(x\ra y),\psi(x\rs y)\}$. \\
It follows that $\psi\in \mathcal{PV}(A)$. \\ 
Suppose that there is another $\psi'\in \mathcal{PV}(A)$ such that $\psi'=\varphi\circ f$. \\
Then, for all $x\in A$, we have $\psi'(x)=(\varphi\circ f)(x)=\psi(x)$, hence $\psi'=\psi$. \\
$(i)$ If $x\in \Ker({\psi})$, then $\psi(x)=0$, so $\varphi(f(x))=0$. \\
Hence $f(x)\in \Ker(\varphi)$, that is $x\in f^{-1}(\Ker(\varphi))$. \\
It follows that $\Ker(\psi)\subseteq f^{-1}(\Ker(\varphi))$. \\
Conversely, let $x\in f^{-1}(\Ker(\varphi)$, that is $f(x)\in \Ker(\varphi)$. 
It follows that $\varphi(f(x))=0$, hence $\psi(x)=0$, that is $x\in \Ker(\psi)$. 
Thus $f^{-1}(\Ker(\varphi))\subseteq \Ker(\psi)$. \\
We conclude that $\Ker(\psi)=f^{-1}(\Ker(\varphi))$. \\ 
$(ii)$ It follows by $(i)$ and Proposition \ref{comm-ds-130-10}. 
\end{proof}

\begin{theo} \label{va-psbe-100}
Let $f: A\longrightarrow B$ be a pseudo-BE isomomorphism and let $\varphi\in \mathcal{PV}(A)$.  
Then there exists a unique $\psi\in \mathcal{PV}(B)$ such that the following diagram is commutative. \\
Moreover $\Ker(\psi)=f(\Ker(\varphi))$.
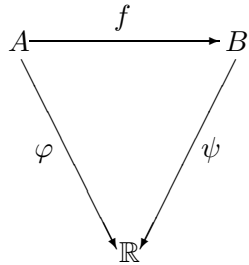
\begin{center}
\begin{picture}(105,100)(0,0)
\put(0,80){$A$}
\put(8,84){\vector(1,0){72}}
\put(82,80){$B$}
\put(40,90){$f$}
\put(42,0){${\mathbb R}$}
\put(5,77){\vector(1,-2){36}}
\put(10,40){$\varphi$}
\put(86,77){\vector(-1,-2){36}}
\put(73,40){$\psi$}
\end{picture}
\end{center}
\end{theo} 
\begin{proof}
Let $\varphi\in \mathcal{PV}(A)$ and let $y\in B$. 
Since $f$ is surjective, there exists $x\in A$ such that $y=f(x)$. Define $\psi(y):=\varphi(x)$. 
Obviously $\varphi(x)=\psi(y)=\psi(f(x))=(\psi\circ f)(x)$. \\ 
Since $f$ is injective, then $1$ is the only element of $A$ such that $f(1)=1$. \\
Hence $\psi(1)=\psi(f(1))=\varphi(1)=0$. \\
Let $y_1, y_2\in B$ and let $x_1, x_2\in A$ such that $f(x_1)=y_1$ and $f(x_2)=y_2$. \\
Then we have: \\
$\hspace*{2cm}$ $\psi(y_2)-\psi(y_1)=\psi(f(x_2))-\psi(f(x_1))=\varphi(x_2)-\varphi(x_1)$ \\
$\hspace*{4.4cm}$ $\le \varphi(x_1\ra x_2)=(\psi\circ f)(x_1\ra x_2)$ \\
$\hspace*{4.4cm}$ $=\psi(f(x_1\ra x_2))=\psi(f(x_1)\ra f(x_2))$ \\
$\hspace*{4.4cm}$ $=\psi(y_1\ra y_2)$. \\
Similarly $\psi(y_2)-\psi(y_1)\le \psi(y_1\rs y_2)$. \\
It follows that $\psi(y_2)-\psi(y_1)\le \min\{\psi(y_1\ra y_2),\psi(y_1\rs y_2)\}$, 
hence $\psi\in \mathcal{PV}(B)$. \\
Suppose that there is another $\psi'\in \mathcal{PV}(B)$ such that $\varphi=\psi'\circ f$. \\ 
Let $y\in B$ and $x\in A$ such that $f(x)=y$. 
Then we have: \\
$\hspace*{2cm}$ $\psi'(y)=\psi'(f(x))=(\psi'\circ f)(x)=\varphi(x)=\psi(y)$, \\
hence $\psi'=\psi$. \\ 
Consider $y\in \Ker(\psi)$, that is $\psi(y)=0$. Let $x\in A$ such that $f(x)=y$.  
Since $\psi(y)=\varphi(x)$, we get $\varphi(x)=0$, that is $x\in \Ker(\varphi)$. 
Hence $y=f(x)\in f(\Ker(\varphi))$, so $\Ker(\psi)\subseteq f(\Ker(\varphi))$. \\
Conversely, let $y\in f(\Ker(\varphi))$, so $y=f(x)$ with $x\in \Ker(\varphi)$, that is $\varphi(x)=0$. \\
It follows that $\psi(y)=\varphi(x)=0$, so $y\in \Ker(\psi)$, hence $f(\Ker(\varphi))\subseteq Ker(\psi)$. \\
We conclude that $\Ker(\psi)=f(\Ker(\varphi))$.
\end{proof}



$\vspace*{5mm}$

\section{Concluding remarks} 

Many information processing branches are based on the non-classical logics and deal with uncertainty information (fuzziness, randomness, vagueness, etc.). For this pupose, different probabilistic models have been constructed on algebras of fuzzy logics: states, generalized states, internal states, state-morphism operators, measures. 
In this paper we show that the commutative property plays an important role in probabilities theory on 
fuzzy structures. We also emphasize that in the case of commutative pseudo-BE algebras the two types 
of internal states coincide. Important results on probabilistic models on algebras of non-classical logic 
have been proved based on fantastic deductive systems.   
The kernel of a Bosbach state (state-morphism, measure, type II state operator, pseudo-valuation) 
on pseudo-BE algebras is a fantastic deductive system. \\
As another direction of research, one could investigate the \emph{$n$-fold commutative pseudo-BE algebras} and 
the \emph{$n$-fold fantastic deductive systems} of pseudo-BE algebras. 
We recall that a BCK-algebra $(A,*,0)$ is called $n$-fold commutative if there exists a fixed natural number $n$ 
such that the identity $x* y=x*(y*(y*x^n))$, where $y*x^0=y$, $y*x^{n+1}=(y*x^n)*x$, holds for all 
$x, y\in A$ (see \cite{Xie1}). 
An ideal $I$ of a BCK-algebra $(A,*,0)$ is called $n$-fold commutative if there exists a fixed natural number $n$ 
such that $x*y\in I$ implies $x* y=x*(y*(y*x^n))\in I$ for all $x, y\in A$ (see \cite{Huang1}). 
State pseudo-BE algebras and state-morphism pseudo-BE algebras could be studied in the context of 
$n$-fold commutative pseudo-BE algebras and $n$-fold fantastic deductive systems. 
These results could be extended to more general structures, such as pseudo-BCI algebras and pseudo-CI algebras.

$\vspace*{5mm}$

\vspace*{3mm}

\begin{flushright}
\begin{minipage}{148mm}\sc\footnotesize
Lavinia Corina Ciungu\\
Department of Mathematics \\
University of Iowa \\
14 MacLean Hall, Iowa City, Iowa 52242-1419, USA \\
{\it E--mail address}: {\tt lavinia-ciungu@uiowa.edu}

\end{minipage}
\end{flushright}

\end{document}